\documentclass[a4paper,12pt,oneside]{article}
\newif\ifdraftmode

\usepackage[margin=1in]{geometry}
\usepackage[utf8]{inputenc}
\usepackage{xargs}
\usepackage[pdftex,dvipsnames]{xcolor}
\usepackage[colorlinks,citecolor=blue,urlcolor=blue,breaklinks=true]{hyperref}
\usepackage{authblk}

\usepackage{graphicx}
\graphicspath{ {./figures/} }

\usepackage{framed}

\usepackage{caption}
\usepackage{subcaption}
\usepackage{enumitem}

\ifdraftmode
    \usepackage[colorinlistoftodos,prependcaption,textsize=tiny]{todonotes}
    \newcommandx{\unsure}[2][1=]{\todo[linecolor=red,backgroundcolor=red!25,bordercolor=red,#1]{#2}}
    \newcommandx{\change}[2][1=]{\todo[linecolor=blue,backgroundcolor=blue!25,bordercolor=blue,#1]{#2}}
    \newcommandx{\info}[2][1=]{\todo[linecolor=OliveGreen,backgroundcolor=OliveGreen!25,bordercolor=OliveGreen,#1]{#2}}
    \newcommandx{\improvement}[2][1=]{\todo[linecolor=Plum,backgroundcolor=Plum!25,bordercolor=Plum,#1]{#2}}
    \newcommandx{\thiswillnotshow}[2][1=]{\todo[disable,#1]{#2}}
    
    \usepackage{seqsplit}
    \usepackage[color,notref]{showkeys}
    \usepackage{xstring}

    \renewcommand*\showkeyslabelformat[1]{%
    \noexpandarg%
    \fbox{\parbox[t]{\marginparwidth}{\raggedright\normalfont\small\ttfamily\expandafter\seqsplit\expandafter{#1}}}}
    \colorlet{refkey}{RubineRed!75}
    \colorlet{labelkey}{RubineRed!75}

    \usepackage{cancel}

    \usepackage{marginnote}
    \usepackage{xpatch}
    \makeatletter
    \patchcmd{\@todonotes@drawMarginNoteWithLine}{\marginpar}{\marginnote}{}{}
    \makeatother
\else
    \newcommandx{\unsure}[2][1=]{}
    \newcommandx{\change}[2][1=]{}
    \newcommandx{\info}[2][1=]{}
    \newcommandx{\improvement}[2][1=]{}
    \newcommandx{\thiswillnotshow}[2][1=]{}
\fi

\usepackage[mode=buildnew]{standalone}
\usepackage{tikz,tikz-cd}
\usetikzlibrary{decorations.pathmorphing} 
\usetikzlibrary{positioning}
\usetikzlibrary{calc}

\usepackage{amsfonts}
\usepackage{amssymb}
\usepackage{amsthm}
\usepackage{amsmath}
\usepackage{bm}
\usepackage[noabbrev,nameinlink,capitalize]{cleveref}

\definecolor{LightGrey}{HTML}{EEEEEE}
\definecolor{Grey}{HTML}{C7C8CC}
\definecolor{DarkGrey}{HTML}{373A40}

\newtheorem{theorem}{Theorem}[section]
\newtheorem{lemma}[theorem]{Lemma}
\newtheorem{proposition}[theorem]{Proposition}
\newtheorem{corollary}[theorem]{Corollary}
\newtheorem{definition}[theorem]{Definition}
\newtheorem{remark}[theorem]{Remark}

\usepackage[ruled,lined,linesnumbered]{algorithm2e}
\DontPrintSemicolon
\newcommand{\Var}{\textup}

\SetCommentSty{mycommfont}
\SetKwComment{tcc}{$\triangleright$~}{}
\SetKwInput{KwRequire}{Require}
\SetNlSty{}{}{:}
\crefname{algocf}{Algorithm}{Algorithms}

\newlist{assumpenum}{enumerate}{1} 
\setlist[assumpenum]{label=(A\arabic*), ref=\theassumption~(\arabic*)}
\Crefname{assumpenumi}{Assumption}{Assumptions}

\DeclareMathOperator*{\argmin}{arg\,min}

\newcommand{\R}{\mathbb R}
\newcommand{\N}{\mathbb N}

\newcommand{\Id}{\mathrm{Id}}

\DeclareMathOperator{\E}{\mathbb E}

\DeclareMathOperator{\Lip}{Lip}

\DeclareMathOperator{\rank}{rank}

\DeclareMathOperator{\Span}{span}

\DeclareMathOperator{\diag}{diag}

\ifdraftmode

\fi

\makeatletter
\newcommand{\proofAppendix}[1]{%
    \ifx\@currentvi\@proofenv
        See~\cref{#1}.
    \else
        \begin{proof}See~\cref{#1}.\end{proof}
    \fi
}
\newcommand*\@proofenv{proof}
\makeatother

\title{Approximation and learning with compositional tensor trains}

\author[1]{Martin Eigel}
\author[2,*]{Charles Miranda}
\author[2]{Anthony Nouy}
\author[1]{David Sommer}
\affil[1]{\footnotesize Weierstrass Institute for Applied Analysis and Stochastics, Berlin, Germany}
\affil[2]{\footnotesize Centrale Nantes, Nantes Universit\'e, Laboratoire de Math\'ematiques Jean Leray UMR CNRS 6629, France}
\affil[*]{\footnotesize Corresponding author: \texttt{charles.miranda@ec-nantes.fr}}
\date{}

\begin{document}
\maketitle

\begin{abstract}
    We introduce compositional tensor trains (CTTs) for the approximation of multivariate functions, a class of models obtained by composing low-rank functions in the tensor-train format. This format can encode standard approximation tools, such as (sparse) polynomials, deep neural networks (DNNs) with fixed width, or tensor networks with arbitrary permutation of the inputs, or more general affine coordinate transformations, with similar complexities.
    This format can be viewed as a DNN with width exponential in the input dimension and structured weights matrices. Compared to DNNs, this format enables controlled compression at the layer level using efficient tensor algebra.
    
    On the optimization side, we derive a layerwise algorithm inspired by natural gradient descent, allowing to exploit efficient low-rank tensor algebra. This relies on low-rank estimations of   Gram matrices, and tensor structured random sketching. Viewing the format as a discrete dynamical system, we also derive an optimization algorithm inspired by numerical methods in optimal control. Numerical experiments on regression tasks demonstrate the expressivity of the new format and the relevance of the proposed optimization algorithms.

    Overall, CTTs combine the expressivity of compositional models with the algorithmic efficiency of tensor algebra, offering a scalable alternative to standard deep neural networks.
\end{abstract}


\ifdraftmode
    \listoftodos
\fi

\section{Introduction}

This work studies compositional tensor trains (CTTs), which are compositions of functions in tensor-train (TT) format, defining a new class of  expressive and compressible models for high-dimensional approximation. We quantify their expressivity by showing that usual approximation tools can be efficiently encoded with CTTs, and propose optimization methods relying on optimal control algorithms (viewing CTTs as a time-discrete dynamical system) or natural gradient schemes, leveraging efficient tensor algebra techniques.

\paragraph{Low-rank tensor formats.} Tensor networks represent a powerful tool in the realm of high-dimensional data analysis and functional approximation, having gained significant attention in recent years due to their ability to represent  efficiently complex structures. Among such formats, TTs~\cite{Oseledets2011} were initially developed to address the curse of dimensionality in quantum physics, where they are also known as Matrix Product States, but have since been applied in a number of other fields, including machine learning and computational science.

\paragraph{Neural networks as compositional formats.} The advent of neural networks (NNs) has transformed our ability to tackle complex problems and became the foundation of modern machine learning. These models are particularly adept at learning intricate patterns and relationships within data. However, as the complexity of tasks increases, so does the necessity for bigger and more sophisticated architectures, which in turn give rise to challenging, time-consuming, and resource-intensive training processes. In fact, typically slow optimization is a central obstacle for achieving SOTA models in current AI topics such as generative modeling and LLMs.

\paragraph{Composition of low-rank formats.} In this evolving landscape of computational techniques, compositional formats have emerged as a promising approach to approximate functions with high accuracy, alleviating the curse of dimensionality under certain conditions. NNs are, of course, the most famous example of such a compositional format. The literature on compositional TTs is still practically non-existent, but there are a few notable works. The first is the Deep Inverse Rosenblatt Transport (DIRT)~\cite{cui2022deep}, which constructs a transport map as a composition of TTs. The second is the tensor rank reduction via coordinate flows~\cite{dektorTensorRankReduction2023}, which builds a nonlinear coordinate transformation so that the function in the new coordinate system has smaller tensor rank, leading to a ridge-like TT. A similar work can be found in \cite{Grelier2019Dec}. 

\paragraph{Optimal control.} A compositional format can be understood as a discretization of a continuous flow. An example of this is the connection between neural ODEs and residual neural networks. This continuous viewpoint offers a natural training procedure by considering the flow and the training loss as state equation and cost functional of an optimal control problem. In a classic optimize-then-discretize approach \cite{onken2021otflowfastaccuratecontinuous}, we can then use classic optimal control techniques like Pontryagin's maximum principle (PMP) to approximate discrete-time optimal controls, which is equivalent to training compositional TTs in our sense.

\paragraph{Contributions of this paper.} This paper investigates the composition of TTs, introduced in \cite{Schneider2024} as a new approximation format, delving into its fundamental principles and examining how they intersect with other approximation tools such as (sparse) polynomials, ridge approximation, deep neural networks, etc. We show how these approximation tools can be encoded into the CTT format, with similar complexities, thereby demonstrating that their approximation spaces are continuously embedded in those of CTTs. This new format comes with efficient and controlled SVD-based compression methods. In addition, this work introduces training procedures of CTTs for general machine learning tasks, based on algorithms inspired from optimal control or natural gradient schemes. These training procedures rely on efficient tensor algorithms that exploit low-rank structures of the layers.

\paragraph{Structure of the paper.} The structure of this work is as follows. \cref{sec:tensor_formats} covers the low-rank formats needed to define our compositional low-rank format, namely the canonical (CP) format and the tensor train (TT) format.
\cref{sec:CTT} introduces the {CTT} format. Aside from rigorous definition and some theoretical properties of the format, we give  examples of  functions  which are extremely difficult to approximate in the TT format due to high ranks, but can be very efficiently expressed in CTT format. \cref{sec:optimization} derives two learning algorithms for the CTT format. The first algorithm, coming from the optimal control literature, is based on Pontryagin's maximum principle (PMP) and the method of successive approximation, whose practical implementation benefits from efficient tensor algebra.  Under certain assumptions on the bases used for defining the CTT format, we can show convergence of the scheme to a unique fixed point satisfying the PMP. The second algorithm is a modified (layer-wise) natural gradient scheme, allowing one to exploit efficient tensor methods and to preserve the low-rank format of the layers. Finally, \cref{sec:numerics} experimentally shows the power of CTTs and the learning methods on test problems.

\section{Tensor Formats}\label{sec:tensor_formats}
We briefly review the TT factorization, its computational costs, and the TT-SVD construction. We also recall how linear maps are represented as TT-matrices and how ranks relate to matricization ranks.


Consider the task of approximating a function from $ \mathbb{R}^d $ to $\mathbb{R}$. Given some finite-dimensional spaces $\mathcal V^{(k)}_{n_k}$ of univariate functions, $1\le k \le d$, we define the finite-dimensional (functional) tensor space $\mathcal V_{\bm{ n}} = \mathcal V^{(1)}_{n_1} \otimes \ldots \otimes \mathcal V^{(d)}_{n_d}$, which is the linear span of elementary tensor products $v^{(1)} \otimes \ldots \otimes v^{(d)} : (x_1, \dots , x_d) \mapsto v^{(1)}(x_1) \ldots  v^{(d)}(x_d) $, with $v^{(k)} \in \mathcal V^{(k)}_{n_k}$. Given bases $\{\phi_i^{(k)}\}_{i=1}^{n_k}$ of the spaces $\mathcal V^{(k)}$, a basis of $\mathcal V_{\bm{ n}}$ is given by the collection of all tensor products $\{\phi^{(1)}_{i_1}\otimes\ldots\otimes \phi^{(d)}_{i_d} : 1\le i_1 \le n_1, \ldots , 1\le i_d \le n_d\}$.


For any function $v$ in $\mathcal V_{\bm{n}}$, there is a unique  algebraic tensor $\bm{A}\in \mathbb{R}^{n_1} \otimes  \ldots \otimes \mathbb{R}^{n_d}$, identified with a multidimensional array in $\mathbb{R}^{n_1\times \ldots \times n_d}$, such that
\begin{equation}\label{eq:tensor_function}
    v = \sum_{i_1,\ldots,i_d=1}^{n_1,\ldots,n_d} \bm{A}({i_1,\ldots,i_d}) \phi^{(1)}_{i_1}\otimes\ldots\otimes \phi^{(d)}_{i_d} ,
\end{equation}
or equivalently 
 \begin{equation}
 \label{eq:tensor_function_frob}
     v(x) = \langle \bm A, \Phi^{(1)}(x_1)\otimes\dots\otimes\Phi^{(d)}(x_d)\rangle_F := \langle \bm A, \bm{\Phi}(x) \rangle_F ,
 \end{equation}
 where $\Phi^{(k)}(x_k) := (\phi_j^{(k)}(x_k))_{j=1}^{n_k}$ and $\langle .,.\rangle_F$ denotes the canonical (Frobenius) inner product. 
In the above notation, $\bm{A}$ is called the \textit{coefficient tensor} of function $v$ with respect to the tensor product basis.



The storage complexity of such a tensor is $ \mathcal O(n^d)$, where $n=\max\{n_1,\ldots,n_d\}$. This exponential scaling, often called the \textit{curse of dimensionality}, makes the treatment of high-dimensional problems intractable for $d \gg 1$. Low-rank \emph{tensor formats} are specifically developed to tackle this \emph{curse of dimensionality} and represent tensors in a storage-efficient way. We mention only two such formats, which are relevant for this work, namely the \textit{canonical format} and the \textit{tensor train} format.

    Like any rank-$1$ matrix $A\in\mathbb{R}^{n_1\times n_2}$ can be defined as the outer product $v\otimes w$ of two vectors $v\in\mathbb{R}^{n_1},w\in\mathbb{R}^{n_2}$, a tensor $\bm{A}\in\mathbb{R}^{n_1\times\ldots\times n_d}$ is said to have rank $1$ if there exist $v^{(1)}\in\mathbb{R}^{n_1},\ldots,v^{(d)}\in\mathbb{R}^{n_d}$ such that 
    \begin{equation}\label{eq:rank_one_tensor}
        \bm{A} = v^{(1)}\otimes\ldots \otimes v^{(d)}.
    \end{equation}
    It is important to note that, while there is only one commonly used notion of  matrix rank, there are several different concepts of ranks of tensors. These include the \textit{canonical polyadic (CP) rank}, the \textit{(hierarchical) Tucker rank} and the \textit{tensor train (TT) rank}, to just name a few. These rank notions are each connected to different decompositions of the underlying tensor, i.e. the \textit{canonical polyadic decomposition}, the \textit{(hierarchical) Tucker decomposition} or the \textit{tensor train} decomposition, much like the matrix rank is connected to a singular value decomposition. The reason we can describe a tensor of the form \eqref{eq:rank_one_tensor} as having general rank $1$ is because the mentioned rank notions 
    (and to the best of our knowledge other important rank concepts) 
    agree on which tensors have rank $1$ and which do not. 

 \paragraph{The canonical format.}
    For any tensor $\bm A \in \mathbb{R}^{n_1\times\ldots\times n_d}$, there exists a $r_{\mathrm{cp}} \in \N$ and a representation
    \begin{equation}\label{eq:cp_decomposition}
        \bm A = \sum_{k=1}^{r_{\mathrm{cp}}} v_k^{(1)} \otimes\dots\otimes v_k^{(d)}
    \end{equation}
    with 
    vectors $\{v_k^{(i)}\}_{k=1}^{r_{\mathrm{cp}}} \subset \mathbb{R}^{n_i}$, for all $i \in \{1,\dots,d\}$. Such a representation of $\bm A$ as a sum of linearly independent rank-$1$ tensors is called a CP decomposition. The minimal number $r_{\mathrm{
    cp}}$ is the CP rank or simply \textit{canonical rank} of the tensor $\bm A$ and denoted by $\rank(\bm A)$. The set of tensors with canonical rank smaller or equal to $r\in \mathbb{N}$ is not closed in general \cite[Example 7.1]{hillar2013most}, which makes the canonical format problematic for optimization, in particular it does not guarantee the existence of a best rank-$r$ approximation. The tensor train format 
    has emerged, among other formats, as a low rank format suitable for optimization. In particular, the set of tensor trains with bounded \textit{tensor train rank} (defined in the following paragraph) is closed. 


    \paragraph{The tensor train format.} 
    Since the (functional) TT format is used primarily in this work, we introduce it in some detail. We start with the appropriate notion of rank. 
     The \textit{tensor train rank} of a tensor $\bm{A}\in \R^{n_1 \times\dots\times n_d}$ is the elementwise smallest tuple $\bm{r}_{\mathrm{TT}}=(r_1,\ldots,r_{d-1})\in\mathbb{N}^{d-1}$ such that there exist matrix valued mappings $U_j : \N_{\leq n_j} \to \R^{r_{j-1} \times r_j}$, $j=1,\ldots,d$, with $r_0=r_d=1$ by convention, such that
    \begin{equation}\label{eq:tt_decomposition}
         \bm A(i_1,\dots,i_d) = U_1(i_1)\dots U_d(i_d) \quad \text{for all } \quad (i_1,\dots,i_d) \in \N_{\leq n_1}\times\dots\times \N_{\leq n_d},
    \end{equation}
    with $\N_{\leq n} := \{1, \dots, n\}$.  
    The decomposition \eqref{eq:tt_decomposition} is called \textit{tensor train decomposition} and the functions $U_j$ are called the \emph{components} of the tensor train. The alternative name of \textit{matrix product states} comes from the fact that each entry of $\bm{A}$ is represented by a product of matrices. Such a representation exists for any tensor, like a singular value decomposition exists for any matrix. In fact, the standard algorithm by which a TT decomposition of a tensor is obtained is called TT-SVD \cite{Oseledets2011}. Note that the components of the tensor train can equivalently be defined as order-3 tensors $\bm U_j \in \R^{r_{j-1} \times n_j \times r_j}$, and therefore, any tensor $\bm A$ can be written in the form
    \begin{equation*}
        \bm A = \sum_{k_0=1}^{r_0}\dots\sum_{k_d=1}^{r_d} \bm U_1(k_0,\cdot,k_1)\otimes\dots\otimes \bm U_d(k_{d-1},\cdot,k_d).
    \end{equation*}
    In the following, we consider the components of a tensor train to be order-3 tensors. To ease notation, we collect the dimensions of the tensor into a tuple $\bm n = (n_1,\ldots,n_d)$. 
    The set of tensors of dimension $\bm n$   with \emph{tensor train rank} $\bm{r}$ is denoted by
    \begin{equation*}
        \mathcal T_{\bm{ n}, \bm{ r}} := \{\bm A \in \R^{n_1\times\dots\times n_d} : \rank_{TT}(\bm A) = \bm{ r}\}.
    \end{equation*}
    For later use, we also define the set of tensors with ranks element-wise smaller or equal to $\bm r$ by $\mathcal T_{\bm{ n}, \leq \bm{ r}} = \bigcup_{0 \le \hat{\bm r}\leq \bm r} \mathcal{T}_{\bm n,\bm \hat{\bm r}}$, which is a closed set.
    Clearly, storing the tensor is equivalent to storing the $d$ order-3 tensors $\bm {U}^{(j)}$. Hence, the storage complexity of a tensor in such a format is in $\mathcal O(dnr^2)$, where $n=\max\{n_1,\ldots,n_d\}$, $r=\max\{r_1,\ldots,r_{d-1}\}$. In the most general case, the rank bound $r$ depends exponentially on the dimension $d$, reflecting the curse of dimensionality for general tensors. However, several notable functions, such as Gaussian potentials, exhibit low rank structures in the TT format \cite{oseledets2013constructive,grelier2022learning,sommer2024generative}, yielding (at most) a polynomial growth of $r$ with $d$. In such cases, the TT format effectively mitigates the curse of dimensionality.
The TT format is often represented with a tensor network diagram as illustrated in \cref{fig:tt}.
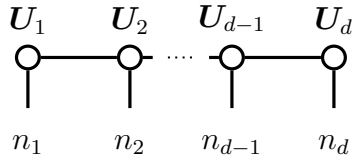
\begin{figure}[h!]
    \centering
    \begin{tikzpicture}[core/.style={
        shape=circle,
        draw=black,
        inner sep=2pt,
        font=\tiny,
        minimum size=0.3cm,
        very thick
    }]
        \def\Leg#1#2{%
            \node[below=0.5 of #1, label={[anchor=north, label distance=-2]below:{{#2}}}] (l#1) {};
        }
        \node[core, label=above:{$\bm{U}_1$}] (V1) {};
        \Leg{V1}{$n_1$}
        \node[core, right=1 of V1, label=above:{$\bm{U}_2$}] (V2) {};
        \Leg{V2}{$n_2$}
        \node[core, right=1 of V2, label=above:{$\bm{U}_{d-1}$}] (Vp) {};
        \Leg{Vp}{$n_{d-1}$}
        \node[core, right=1 of Vp, label=above:{$\bm{U}_d$}] (Vd) {};
        \Leg{Vd}{$n_d$}
    
        \draw[very thick] (V1) edge (lV1) -- (V2) edge (lV2)
                                (V2) -- ($(V2)+(0.3,0)$)
                                ($(Vp)-(0.3,0)$) -- (Vp)
                                (Vp) edge (lVp) -- (Vd) edge (lVd);
        \draw [line width=0.75, dotted] ($(V2)+(0.5,0)$) -- ($(Vp)-(0.5,0)$);
    \end{tikzpicture}
    \caption{Tensor-train format with cores $\bm{U}_1,\dots,\bm{U}_d$}
    \label{fig:tt}
\end{figure}

Interestingly, the space of tensors $\mathcal T_{\bm n, \bm r}$ with fixed TT ranks is a Riemannian submanifold.
    The dimension of this manifold is $
        \dim{\mathcal T_{\bm{ n}, \bm{ r}}} = \sum_{i=1}^d r_{i-1}n_i r_i - \sum_{i=1}^{d-1} r_i^2$
\cite[Lemma~4.1]{Holtz2011}. Note that the first sum in this expression is just the added up storage complexity of all $d$ components. The second term results from a \textit{gauge invariance} inherent in the format that can be seen as follows. For any two subsequent components $U_j, U_{j+1}$ we can transform $U_j(\cdot)\to U_j(\cdot)R^{-1}$ and $U_{j+1}(\cdot)\to RU_{j+1}(\cdot)$, where $R\in\mathbb{R}^{r_j\times r_j}$ is an invertible matrix, without changing the tensor. Hence, the components $U_j$ are only unique up to such invertible transformations, which is taken into account by the second sum making up $\dim(\mathcal{T}_{\bm n, \bm r})$. 

Any tensor in CP format can easily be represented in TT format. Indeed, let $\bm A = \sum_{k=1}^{r_{\mathrm{cp}}} v_k^{(1)} \otimes\dots\otimes v_k^{(d)}$, $v_k^{(i)} \in \R^{n_i}$. Then
\begin{equation*}
    \bm A(i_1,\dots,i_d) = \begin{bmatrix}
        v_1^{(1)}(i_1) & \dots & v_{r_{\mathrm{cp}}}^{(1)}(i_1)
    \end{bmatrix}
    \begin{bmatrix}
        v^{(2)}_1(i_2) & & 0\\
         & \ddots & \\
        0 & & v^{(2)}_{r_{\mathrm{
        cp}}}(i_2)
    \end{bmatrix}
    \dots
    \begin{bmatrix}
        v_{{1}}^{(d)}(i_d)\\
        \vdots\\
        v_{r_{\mathrm{cp}}}^{(d)}(i_d)
    \end{bmatrix}
\end{equation*}
is a representation in TT format, which can easily be verified by multiplying out the matrix product. Note that the CP rank $r_{\mathrm{cp}}$ is a uniform upper bound for the TT ranks. Rounding to a prescribed rank in the TT format is done via the procedure described in \cite[Algorithm 2]{Oseledets2011}, i.e. by successively truncating the singular value decompositions of component unfoldings to the desired rank (or accuracy). For later use, we let $\mathtt{round}_{\bm r}$ denote the rounding procedure of a tensor in $\mathbb{R}^{n_1\times\ldots\times n_d}$ to fixed rank $\bm r$.

\paragraph{Functional tensor trains.} 
The set of functions in $\mathcal{V}_{\bm n}$ with coefficient tensors in $\mathcal{T}_{\bm n, \bm r}$ is denoted $\mathcal{T}_{\bm r}(\mathcal{V}_{\bm n})$ and called the \emph{functional tensor train format}. A function $v \in \mathcal{T}_{\bm r}(\mathcal{V}_{\bm n})$ admits the representation  
\begin{equation}\label{eq:ftt}
    v(x_1,\ldots, x_d) = \hat{U}_1(x_1) \hat{U}_2(x_2) \cdots \hat{U}_d(x_d)
\end{equation}
with matrix valued functions $\hat{U}_i(x_i)\in\mathbb{R}^{{r}_{i-1} \times {r}_i}$ for $i=1,\ldots,d$, again with the convention ${r}_0={r}_d=1$. The function $\hat{U}_i$ is identified with a third-order tensor $\bm U_i \in \mathbb{R}^{r_{i-1} \times n_i \times r_i}$ through the relation  $\hat {U}_i(k_{i-1},x_i,k_i) = \sum_{j=1}^{n_i} \bm U_i(k_{i-1},j,k_i)\phi_j^{(i)}(x_i)$. 

The smallest $\bm r$ such that $v$ is in $\mathcal{T}_{\bm r}(\mathcal{V}_{\bm n})$ is the tensor train rank of $v$, which coincides with the tensor train rank of its coefficient tensor, independently of the choice of bases of functions spaces $\mathcal{V}^{(i)}_{n_i}$. Results on  approximation with TT  format of functions of classical smoothness classes or with compositional structures can be found in  \cite{Schneider201456,griebel2023low,bachmayr2023low,bachmayr2023approximation,ali2021approximation,ali2023approximation}.

\paragraph{Vector valued tensor trains.}
In this work, we are frequently concerned with vector valued functions $v\colon \mathbb{R}^d \to \mathbb{R}^d$. For this case, we permit $r_0=d$, leading to a first component $\bm U_1\in\mathbb{R}^{d\times n_1\times {r}_1}$ such that the contraction of all components yields a vector in $\mathbb{R}^d$. This first component tensor is identified with the component function  $\hat U_1 \in \mathbb{R}^{d} \otimes \mathcal{V}^{(1)}_{n_1} \otimes \mathbb{R}^{r_1}$. A particular case of interest are functions of the form $v\colon \mathbb{R}^d\to\mathbb{R}^d, v(x) = (v_1(x),0,\ldots,0)^{\top}$ which map to zero except in the first output dimension. In this case, a functional TT representation is obtained by taking a (scalar valued) TT representation of $v_1$ and padding the first component of this representation with zeros for the other outputs. The TT representation of $v$ inherits the ranks of $v_1$ in this case (except for $r_0$), since all other components remain the same.

\begin{lemma}\label{lem:vTT_firstcomp}
    Assume $v\colon \mathbb{R}^d \to \mathbb{R}^d$, $d\in\mathbb{N}$, has the form $v(x) = (v_1(x),0,\ldots,0)^{\top}$ for a $v_1\colon \mathbb{R}^d\to\mathbb{R}$ with TT rank $\bm r = (r_1,\ldots,r_{d-1})$ and $r_0=r_d=1$. Then, $v$ can be represented by a TT with rank $\bm r$ and $r_0=d$, $r_d=1$.
\end{lemma}
\begin{proof}
    Assume the TT representation of $v_1$ is $v_1(x) = \hat{U}^{(1)}(x_1) \ldots \hat{U}^{(d)}(x_d)$ with $\hat{U}^{(i)}(x_i)\in\mathbb{R}^{r_{i-1}\times r_i}$, $r_0=r_d=1$. Define $\tilde{U}^{(i)}=\hat{U}^{(i)}$ for $i=2,\ldots,d$ and 
    \begin{equation*}
        \tilde{U}^{(1)}(x_1) = \begin{bmatrix}
            \hat{U}^{(1)}(x_1) \\ \bm 0 \\\vdots\\\bm 0
        \end{bmatrix} \in \mathbb{R}^{d\times r_1}.
    \end{equation*}
    Then, $v(x) = \tilde{U}^{(1)}(x_1)\ldots \tilde{U}^{(d)}(x_d)$ is a TT representation of $v$, which preserves all ranks, except $r_0$.
\end{proof}
This result enables us to represent vector-valued functions, where only the first output dimension is nonzero, as TTs. For later use, we state another immediate result regarding the rank of such functions with additional structure.

\begin{lemma}\label{lem:vtt_product}
    Assume $v\colon \mathbb{R}^{d +1} \to \mathbb{R}^{d +1}$, $d\in\mathbb{N}$, has the form $v(x) = (v_{2:d+1}(x_{2:d+1}) v_1(x_1),0_d)^{\top}$, where $x_{2:d+1}=(x_2,\ldots,x_{d+1}) $, $0_d$ is the zero vector in $\mathbb{R}^d$, $v_1\colon \mathbb{R}\to\mathbb{R}$, and $v_{2:d+1}\colon \mathbb{R}^{d}\to\mathbb{R}$ has TT rank $\bm r^* = (r^*_1,\ldots,r_{d-1}^*)$. 
    Then, $v$ can be represented by a TT with rank 
    \begin{equation*}
    \bm r = (r_1,\ldots,r_{d}) = (1,r^*_1,\ldots,r^*_{d-1})
    \end{equation*}
    and $r_0=d+1$, $r_{d+1}=1$.
\end{lemma}
\begin{proof}
    Let $v_{2:d+1}(x_{2:d+1})=\hat{U}_1(x_2)\ldots \hat{U}_{d+1}(x_{d+1})$ with $\hat{U}_i(x_i)\in\mathbb{R}^{r^*_{i-1}\times r^*_i}$ be a TT representation of $v_{2:d+1}$. Then, $v_1(x_1) \hat{U}_1(x_2)\ldots \hat{U}_d(x_{d+1})$ is a TT representation of $v_1(x_1) v_{2:d+1}(x_{2:d+1})$. Since $v_1(x_1)\in\mathbb{R}^{1\times 1}$, the rank of this representation is $(1,r^*_1,\ldots,r^*_{d})$. The claim now follows with \cref{lem:vTT_firstcomp}.
\end{proof}

\section{Compositional Tensor Trains (CTT)}\label{sec:CTT}
TTs provide efficient representations for many high-dimensional functions, but their performance depends critically on a favourable ordering of variables and the presence of intrinsic low-rank structure. When these conditions are violated, TT ranks can grow rapidly and the format becomes ineffective.

This section develops \emph{Compositional Tensor Trains (CTT)} as a new representation model and establishes their efficiency and approximation power.
We first motivate composition by exhibiting functions that are provably awkward for plain TT yet become efficiently encoded via CTT.
We then formalize the new CTT format and show how standard function classes are encoded with controlled ranks (affine maps, univariate and multivariate polynomials, concatenation of functions), as well as classical deep neural networks.
The representation allows for the derivation of a universal approximation property for CTT under mild assumptions on the feature basis $\bm{\Phi}$.
Moreover, based on an error-propagation analysis, we give a compression algorithm of the layers with guarantees on the precision.



\subsection{Where TTs fail}\label{sec:TTs_fail}
We justify the compositional design by constructing targets whose coefficient tensors have large intermediate TT ranks, while the same targets admit low-rank \emph{compositions} of simple Euler layers, thus exposing the depth advantage exploited by CTT.

\subsubsection{Markov processes}\label{sec:markov} As a first example where TTs can fail, consider a discrete time Markov process $X = (X_1, \dots, X_d)$, as in \cite[Section 3.2.2]{grelier2022learning} whose density is given by
\begin{equation*}
    f(x) = f_{d|d-1}(x_d|x_{d-1}) \dots f_{2|1}(x_2|x_1) f_1(x_1),
\end{equation*}
where $f_1$ is the density of $X_1$ and $f_{i|i-1}(\cdot | x_{i-1})$ is the conditional density of $X_i$ knowing $X_{i-1} = x_{i-1}$. Let $m_i := \rank(f_{i|i-1})$. Then, the TT ranks of $f$ are
\begin{equation*}
    \bm{r} = (1, m_2, m_3, \dots, m_d, 1).
\end{equation*}
If the variables are reordered, for example  $\tilde{x} = (x_{\sigma(1)}, \dots, x_{\sigma(d)})=: P_\sigma(x)$ for some permutation $\sigma$, the TT ranks of the function $\tilde f(\tilde x) = f(P_\sigma^{-1}(\tilde x))$ can grow exponentially with the dimension $d$. For instance, consider the permutation
\begin{equation*}
    \sigma = \left(1, 3, 5, \dots, 2\left\lfloor\frac{d+1}{2}\right\rfloor -1, 2, 4, 6, \dots 2\left\lfloor\frac{d}{2}\right\rfloor\right)  
\end{equation*}
which places all odd-indexed variables first, followed by the even-indexed variables. The TT ranks then become
\begin{equation*}
    \begin{aligned}
        \tilde{r}_k = \begin{cases}
            \prod_{j=2}^{\min(2k,d)} m_j, &1 \leq k \leq \lceil d/2 \rceil,\\
            \prod_{j=2(k-\lceil d/2 \rceil)+2}^d m_j, &\lceil d/2 \rceil + 1 \leq k \leq d,
        \end{cases}
    \end{aligned}
\end{equation*}
which grow exponentially with $d$. This growth occurs because the Markov property is local: each variable depends only directly on its immediate predecessor. When variables are reordered non-sequentially, dependencies become effectively non-local in the TT decomposition, requiring large ranks to capture the induced correlations.

We will see later that CTTs can encode such functions efficiently, with one layer encoding the permutation map $P_\sigma^{-1}: \mathbb{R}^d\to \mathbb{R}^d$ (see Section \ref{sec:encoding_linear_maps}), and one layer encoding $f$ in TT format. 

\subsubsection{Gaussian densities}\label{sec:gaussian_theory} 

As a second example, consider  the case of the Gaussian function
\begin{equation}
    f_\Gamma(x) := \exp\left(-\frac 1 2 x^\top \Gamma x\right),
    \label{eq:gaussian}
\end{equation}
where $\Gamma\in\mathbb{R}^{d\times d}$ is a symmetric positive definite precision matrix. 
In the trivial case where $\Gamma = \diag(\gamma_1,\dots,\gamma_d)$ is diagonal, $f_\Gamma$ immediately factorises to
$
    f_\Gamma(x) = \prod_{i=1}^d \exp\left(-\frac 1 2 \gamma_i x_i^2 \right), 
$
which is a rank 1 function. However, when the precision matrix $\Gamma$ is non-diagonal, the situation changes drastically, depending on the matrix ranks of the sub-diagonal blocks.
The approximation of $f_\Gamma$ by TTs is studied in~\cite{rohrbach2022rank}.
For example, the average maximal TT rank required to approximate a $15$-dimensional Gaussian with sub-diagonal precision blocks of rank $4$ up to a relative $L^2$ accuracy of $10^{-4}$ is experimentally shown to be close to $100$.
In the most general case, the ranks of the sub-diagonal blocks grow linearly in $d$, and the TT ranks to achieve an accuracy $\varepsilon$ grow exponentially in $d$. 

\paragraph{The flow perspective.} As we show now, compositions of tensor trains is a  better suited format to approximate a Gaussian. First, we define a \textit{lift} from $\R^d$ to $\R^{d+1}$ by 
$
    \mathfrak{L}(x) = \begin{pmatrix}
        1 \\ x
    \end{pmatrix}.
    $
Second, we define a flow $\Phi_t : \R^{d+1} \to \R^{d+1}$ by
\begin{equation}\label{eq:cont_flow}
\begin{aligned}
    \partial_t \Phi_t(h) &= \begin{bmatrix}
        (-\frac 1 2 \Phi_t(h)_{2:d+1}^\top \Gamma \Phi_t(h)_{2:d+1})\Phi_t(h)_1& 0_d
    \end{bmatrix},\\
    \Phi_0(h) &= h
\end{aligned}
\end{equation}
for any $h\in \R^{d+1}$. Choosing $h = \mathfrak{L}(x)$, we get
\begin{equation*}
    \Phi_1(\mathfrak{L}(x)) = \begin{bmatrix}
        f_\Gamma(x) & x
    \end{bmatrix}.
\end{equation*}
Indeed, it can be easily shown that for $0 \leq t \leq 1$, we have $\Phi_t(\mathfrak{L}(x)) = \begin{bmatrix}
    \exp\left(-\frac{1}{2} t x^\top \Gamma x\right) & x
\end{bmatrix}$.
We recover the Gaussian function by simply using the projection onto the first variable, i.e. $\mathfrak{R}\colon \R^{d+1}\to \R, x\mapsto x_1$, leading to 
\begin{equation*}
    f_{\Gamma}(x) = \mathfrak{R}\circ \Phi_1 \circ \mathfrak{L}(x).
\end{equation*}
%
%
%
%
While a similar construction can be used to represent much more general functions by continuous flows, we want to focus on the Gaussian here. The reason for this is that the vector field on the right hand side of \eqref{eq:cont_flow} has a provably low TT rank, as we show in the following.

\paragraph{Flows with provably low rank.} First, note that by \cref{lem:vtt_product}, the TT rank of the right-hand side of \eqref{eq:cont_flow} is $(1, \bm r)$, where $\bm r$ is the TT rank of the quadratic form $\frac{1}{2}y^\top \Gamma y$, which is bounded by $\frac{d}{2}$ \cite[Lemma D.1]{sommer2024generative}.
By using polynomial basis functions up to degree $2$, i.e. $\Phi=\{1, x, x^2\}$, we can get an exact TT representation of the right-hand side with a  complexity in $\mathcal{O}(d^3)$.



\paragraph{Time discretization of the flow.} The flow \eqref{eq:cont_flow} can be generalized to the form
\begin{equation}\label{eq:general_flow}
    \partial_t \Phi_t(h) = f_t(\Phi_t(h)),
\end{equation}
where $f_t\colon \mathbb{R}^{d+1}\to\mathbb{R}^{d+1}$ for each $t\in[0,1]$, permitting explicit time dependence of $f$. Any time discretization of \eqref{eq:general_flow}, e.g. by Euler or Runge-Kutta methods, leads to a compositional structure in terms of $f_t$. For a number of intervals $N\in\mathbb{N}$ and time points $t_0=0,t_n=\frac{n}{N}$, $n=1,\ldots,N$ such methods compute  approximations $\widehat{\Phi}^N(h) \approx \Phi_1(h)$ by various forms of compositions of $f_t$. For two obvious examples, with $\tau=\frac{1}{N}$, Euler's method is defined by $\widehat{\Phi}^N=(I+\tau f_{t_{N-1}})\circ (I+\tau f_{t_{N-2}}) \circ \ldots \circ (I+\tau f_{t_1}) \circ (I+\tau f_{0})$, whereas Heun's method is defined by $\widehat{\Phi}^N = (I+\frac{\tau}{2} (f_{t_{N-1}} + f_{t_N}\circ(I + \tau f_{t_{N-1}}))) \circ \ldots$
If the right-hand side of \eqref{eq:general_flow} is a TT as is the case of the flow \eqref{eq:cont_flow}, such a discretization immediately defines an approximation of the function $\Phi_1(h)$ as a compositional tensor train (CTT). For many target functions it is not clear whether a suitable flow of the form \eqref{eq:general_flow} exists and if it does, if the right-hand side $f_t$ has low-rank structure. For such cases, we need to define a general CTT architecture as well as a method to train it given the target in a way that it can adaptively uncover potential low rank structures.


\subsection{Compositional architectures}\label{sec:comp_architectures}

In the following we introduce the CTT architecture proposed in this work. It is inspired by time discretization of low-rank flows such as for the Gaussian in \eqref{eq:cont_flow}. Furthermore, as we show in the next section, it comes with  optimization procedures exploiting efficient tensor algebra.




\begin{definition}[Compositional tensors]\label{def:flow_CTT}
    Given linear operators $\mathfrak{L} : \R^d \to \R^p$ and $\mathfrak{R} : \R^p \to \R^{d_o}$, with $d,p,{d_o} \in \N$ and $p \geq d$, called respectively \emph{lift} and \emph{retraction}, we denote a function $v : \R^d \to \R^{d_o}$ a \emph{compositional tensor (CT)} with $L \in \N$ layers and univariate basis $\Phi := \{\phi_j : \R \to \R\}_{j=1}^n$ with $n \in \N$ if there exist functional tensors $\psi_1,\dots,\psi_L$ from $\R^p$ to $\R^p$, each with the same univariate basis $\Phi$ and a coefficient tensor $\bm\psi_k\in\R^{p\times n\times\ldots\times n}$, defined by
    \begin{equation}
        \psi_k : x \mapsto \Big( \sum_{i_1,\ldots,i_p}\bm \psi_{k}(j,i_1,\ldots,i_p)\phi_{i_1}(x_1)\dots \phi_{i_p}(x_p)\Big)_{1\le j \le p},
        \label{eq:coefficient_tensor}
    \end{equation}
    such that
    \begin{equation}
        v(x) = \mathfrak{R} \circ (\Id + \psi_L) \circ \dots \circ (\Id + \psi_1) \circ \mathfrak{L}(x), \quad x \in \R^d.
        \label{eq:ctt}
    \end{equation}
    Additionally, if the tensors $\psi_1,\dots,\psi_L$ have TT-ranks $\bm{r}_1,\dots,\bm{r}_L$ respectively, we call $v$ a \emph{compositional tensor-train (CTT)} with ranks $\bm{r}_1,\dots,\bm{r}_L$.

    The set of \emph{compositional tensors} with basis $\Phi$ and with $L$ layers is denoted $\mathcal{CT}^L(\Phi; \mathfrak{L}, \mathfrak{R})$. If the tensors have TT-ranks bounded by $\bm{r}$, then the corresponding set is denoted $\mathcal{CTT}_{ \bm{r}}^L(\Phi; \mathfrak{L}, \mathfrak{R})$.
\end{definition}
In the following, for simplicity we omit $\mathfrak{L}$ and $\mathfrak{R}$ in the notations and write
\begin{equation*}
    \mathcal{CT}(\Phi) := \bigcup_{L \geq 1} \mathcal{CT}^L(\Phi), \qquad \mathcal{CTT}_{ \bm{r}}(\Phi) := \bigcup_{L \geq 1} \mathcal{CTT}_{ \bm{r}}^L(\Phi).
\end{equation*}

\paragraph{CTTs as  DNNs.}
Note that CTTs and DNNs are closely connected. Indeed, a function $u$ in CTT format admits a recursive representation, letting 
 $u_0(x) =  \mathfrak{L}(x)$, 
\begin{align*}
    u_{k+1}(x) = u_k(x) + \psi_k(u_k(x)), \quad 0\le k \le L-1,
\end{align*}
and $u(x) = \mathfrak{R}(u_L(x))$. 
Assuming $1,\Id \in \Phi$ and with the corresponding feature map $\bm{\Phi} : \mathbb{R}^p \to \mathbb{R}^{p \times n\times \dots \times n}$, we can introduce a linear map $\bm{E} : \mathbb{R}^{pn^p} \to \mathbb{R}^p$ such that $\bm{E} \operatorname{vec}(\bm{\Phi}(y)) = y$. Then letting 
$v_k(x) = \operatorname{vec}(\bm{\Phi}(u_k(x))) \in 
 \mathbb{R}^{p n^p}$
 and   
$\psi_k(y) = \langle \bm{\psi}_k , \bm{\Phi}(y)\rangle_F$, 
we have that 
$u(x) = \mathfrak{R}( \bm{E}  v_L(x) )$ with the recursion 
$$
v_{k+1}(x) = \bm{\Phi}( \bm{A}_k v_k(x)), \quad 0\le k \le L-1, $$
where $\bm{A}_k = \bm{E} + \bm{B}_k$, with a linear map $\bm{B}_k :\mathbb{R}^{pn^p} \to \mathbb{R}^p$ such that 
$ \langle \bm{\psi}_k , \Phi(y) \rangle_F = \bm{B}_k \operatorname{vec}(\bm{\Phi}(y)) $.
Therefore, a CTT with basis $\Phi$ can be seen as a DNN with 
a latent space of very high dimension $p n^p$, highly structured linear maps $\bm{A}_k : \mathbb{R}^{pn^p} \to \mathbb{R}^p$ and 
nonlinear activation function $y\mapsto \operatorname{vec}(\bm{\Phi}(y)) $ associated with a tensorization of the chosen basis $\Phi$.

\subsection{Encoding of classical function classes} 

\emph{Compositional tensors} possess the ability to exactly encode specific classes of functions, a property that arises not only from their inherent compositional architecture but also from the structured tensor representation employed at each layer of the model. This allows representation of complex mappings in terms of simpler functions. In particular, the expressive power of these tensors depends crucially on the choice of the basis function $\Phi$. By selecting an appropriate basis (also called features and hence feature space, respectively) such as $\{1, x\}$ or $\{1,\operatorname{ReLU}\}$, the compositional tensor can represent exactly approximation tools such as polynomials or neural networks.

In this section, we provide explicit constructions that illustrate how compositional tensors can encode these function classes exactly. We begin by giving explicit encoding of functions such as linear maps, ridge function, and polynomials. Then, we show that with an appropriate choice of $\Phi$, we can also exactly represent neural networks. This explicit encoding highlights both the theoretical expressivity and the practical utility of compositional tensors in approximating high-dimensional functions.

\subsubsection{Linear maps, coordinate transformations}\label{sec:encoding_linear_maps}

The subsequent result shows that linear affine maps can be efficiently encoded in the TT format. 

\begin{proposition}[Linear affine maps written in the tensor-train format]\label{proposition:linear-transformation}
    The linear map $x \in \R^d \mapsto \langle a, x\rangle$ with $a\in \R^d$ is an additive function, and therefore can be rewritten as a tensor train with ranks $(2,\ldots, 2)$.  For $\bm A \in \R^{m\times d}$, $b\in \R^m$, the linear affine map $x \in \R^d \mapsto \bm A x + b\in \R^{m}$ has a tensor train representation with ranks at most $d$. 
\end{proposition}
\begin{proof}
The first claim results from the fact that the $k$-th unfolding of $\langle a, x\rangle =\sum_{i=1}^k a_i x_i +\sum_{i=k+1}^d a_i x_i  $ has rank $r_k = 2$, for all $1\le k \le d-1$.
Letting $a_i$ be the $i$-th column of $\bm A$, a suitable representation is given by 
$\hat U_1(x_1) = \begin{pmatrix} x_1 a_1 + b & a_2 & \ldots & a_d  \end{pmatrix}\in \R^{m\times d}$, $\hat U_k(x_k) = \bm I_d + (x_k-1)e_k e_k^\top = \text{diag}(1, \ldots, x_k , \ldots, 1) \in \R^{d \times d}$ for $2\le k \le d-1$ and $\hat U_d(x_d) = (1, \ldots, 1, x_d)^\top \in \R^{d \times 1 }$.
\end{proof}
Since the above result also holds for the linear map {$h\mapsto -h + \bm A h$}, we get the following corollary.

\begin{corollary}\label{cor:linear_trafo_lift}
    Any linear affine transformation $h \in \R^p\mapsto \bm Ah + b \in \R^p$ in the lifted space, with $\bm A\in \R^{p\times p}$ and $b\in \R^p$  can be represented by a single CTT layer with rank at most $p$.
\end{corollary}

The above result shows that a CTT format with $2$ layers can represent functions of the form $u(\bm A  x)$, with $u$ represented in TT format.  
Consequently, a CTT can encode linear coordinate transformations. 
This solves a main difficulty of the TT format (mentioned in Section \ref{sec:markov}), which  is the choice of a good ordering of the variables, usually addressed using prior information on the function ~\cite{dektorTensorRankReduction2023} or stochastic optimization procedures \cite{GNC19,grelier2022learning,Michel22}.

\subsubsection{Polynomials}

The following result shows that univariate polynomials can be represented exactly, provided we use a suitable lift. 
\begin{proposition}
    Let $d \in \N$, and $(a_0, \dots, a_d) \subset \R$. We consider the polynomial
    \begin{equation*}
        g : x \mapsto a_0 + a_1 x + \dots + a_d x^d.
    \end{equation*}
    Then,  letting $\mathfrak{L} : x \mapsto (0, x)$, $\mathfrak{R} : (x,y) \mapsto x$ and
    \begin{align*}
        f_1 : (x,y) &\mapsto (a_d, 0),\\
        f_k : (x,y) &\mapsto (a_{d-k+1} + x(y-1), 0), \quad 2 \le k \le d+1,
    \end{align*}
   it holds that
    \begin{equation*}
        \mathfrak{R} \circ (\Id + f_{d+1}) \circ \dots \circ (\Id + f_1) \circ \mathfrak{L} = g.
    \end{equation*}
    This shows that any \emph{univariate polynomial} can be represented by a \emph{compositional tensor} with $d+1$ layers and with the basis $\Phi=\{1, \Id\}$.
\end{proposition}
\begin{proof}
    The proof is based on the Horner's method which rewrites the evaluation $g(x)$ as the recursion
    \begin{align*}
        b_d &= a_d, \quad b_{k} = a_{k} + b_{k+1}x, \quad k=d-1, d-2, \dots, 1.
    \end{align*}
    To show that indeed $\mathfrak{R} \circ (\Id + f_{d+1}) \circ \dots \circ (\Id + f_1) \circ \mathfrak{L} = g$, we only have to show that for any $k \geq 1$ and $x \in \R$, we have
    \begin{equation*}
        \langle (\Id + f_k) \circ \dots \circ (\Id + f_1) \circ \mathfrak{L}(x), e_1\rangle = b_{d-k+1}.
    \end{equation*}
    Then, since $\mathfrak{R}$ is only the projection onto the first variable, the result follows.

    First, for $k=1$ we have
    \begin{align*}
        (\Id + f_1) \circ \mathfrak{L}(x) &= (0,x) + (a_d, 0)= (a_d, x)= (b_d, x).
    \end{align*}
    Next, suppose that we have $(\Id + f_k) \circ \dots \circ (\Id + f_1) \circ \mathfrak{L}(x) = (b_{d-k+1}, x)$ for $1 \leq k \leq p$. Then, since $f_{k+1} : (x,y) \mapsto (a_{d-k} + x(y-1), 0)$, we have
    \begin{align*}
        (\Id + f_{k+1}) \circ \dots \circ (\Id + f_1) \circ \mathfrak{L}(x) &= (a_{d-k} + b_{d-k+1}x, x)= (b_{d-k}, x),
    \end{align*}
    which concludes the proof.
\end{proof}

Exploiting the fast exponentiation offered by the compositional format, we give an explicit encoding of a (sparse) multivariate polynomial. Let
\begin{equation*}
    P(x_1,\dots,x_d) = \sum_{\alpha \in \Lambda} C_\alpha x^\alpha, \quad x^\alpha = \prod_{j=1}^d x_j^{\alpha_j},
\end{equation*}
where $\Lambda \subset \N^d$ contains the set of multi-indices corresponding to non-zero coefficients $C_\alpha$ in the monomial basis. Define the maximal degree for each variable $x_j$ by
\begin{equation}\label{eq:Nj}
    N_j := \max \{\alpha_j : \alpha \in \Lambda \} \in \N, \quad j = 1\dots d.
\end{equation}
The case $\Lambda = \emptyset$ is trivial, and we do not consider it.

\begin{proposition}[Encoding of multivariate polynomials]\label{proposition:encoding-multivariate-polynomial}
    Let $P$ be a multivariate polynomial in $d$ variables, $\Lambda \subset \N^d$ the set of multi-indices associated with its non-zero coefficients. Suppose that $\Phi = \{1, \Id\}$. Let $q \geq 2$ be the number of \emph{workspace slots} and $p = d+2+q$.  Then, there exists a lift $\mathfrak{L} : \R^d \to \R^p$ and a retraction $\mathfrak{R} : \R^p \to \R$, respectively defined by 
    \begin{align*}
            \mathfrak{L}(x) &= (0,x_1,\dots,x_d,0,\dots,0), \quad 
        \mathfrak{R}(y) = y_1,
    \end{align*}
    and $L$ functions ${\psi}_1,\dots \psi_L$ in TT format such that 
     \begin{align*}
        P &= \mathfrak{R} \circ (\Id + \psi_L) \circ \dots \circ (\Id+\psi_1) \circ \mathfrak{L},\end{align*}
        with 
        \begin{align*}
        L &= |\Lambda|(2+d) + \sum_{\alpha \in \Lambda}\sum_{j=1}^d \lfloor \log_q(\alpha_j+1)\rfloor =\mathcal{O}\left(d|\Lambda|\lfloor \log_q(N_{\max}+1)\rfloor\right),
        \end{align*}
    where $N_{\max} := \max_j N_j$.
    Moreover, the number of non-zero parameters in the encoding is $\mathcal{O}(|\Lambda|d[1 + \lfloor \log_q(N_{\max}+1) \rfloor])$ and the number of parameters is $\mathcal{O}(L(d+q))$. The TT ranks of the functions ${\psi}_\ell$ are bounded by $\bm{r}=(p,2,\dots,2)$.
\end{proposition}
\begin{proof}
    We decompose the state variable $y \in \R^{p}$, with $p=d+2+q$, as
    \begin{equation*}
        y = (\underbrace{y_1}_{\text{accumulator}}, \underbrace{y_2,\dots, y_{d+1}}_{\text{variables}}, \underbrace{y_{d+2}}_{\text{register}}, \underbrace{y_{d+3}, \dots, y_{d+2+q}}_{\text{workspace}}).
    \end{equation*}
    The idea of the encoding is to iterate over the $\alpha \in \Lambda$ and construct iteratively the polynomial $P$ by adding monomials successively to the accumulator $y_1$. For doing so, we loop over all the variables $x_j$, compute the powers $x_j^{\alpha_j}$ and multiplying to the register $y_{d+2}$ in order to obtain the monomial $c_\alpha \prod_{j=1}^d x_j^{\alpha_j}$. In order to compute these powers, we use the ``workspace slots'' $y_{d+3},\dots,y_{d+2+q}$ and  \emph{fast exponentiation}, which computes $a^n$, for some number $a$ and power $n$, in a recursive manner and using only $\mathcal{O}(\log n)$ operations. Indeed, for $\alpha \in \Lambda$, the base-$q$ representation of $\alpha_j$, for $1 \leq j \leq d$, is given by
    \begin{equation*}
        \alpha_j = \sum_{k=0}^{K_{\alpha,j}-1} \alpha_{j,k} q^k, \quad \alpha_{j,k} \in \{0, \dots, q-1\},
    \end{equation*}
    where $K_{\alpha,j} := \lfloor \log_q(\alpha_j + 1)\rfloor$. Next, when raising $x_j$ to the power $\alpha_j$ we get
    \begin{equation*}
        x_j^{\alpha_j} = x_j^{\sum_{k=0}^{K_{\alpha,j}-1} \alpha_{j,k} q^k} = \prod_{k=0}^{K_{\alpha,j}-1} \left(x_j^{q^k}\right)^{\alpha_j,k}.
    \end{equation*}
    We give below the encoding of $P$ in the form of an algorithm, counting the number of non-zero parameters and TT ranks.
    
    \begin{algorithm}[H]
        Store $y_{d+2} \gets c_\alpha$ \tcc*{1 layer, 2 non-zero parameters, TT ranks $\bm{r}=(p, 2, \dots, 2)$}\;
        \For{$j \gets 1$ \KwTo $d$}{
            Store $y_{d+3},\dots,y_{d+2+q} \gets x_j$ \tcc*{1 layer, 2 non-zero parameters, TT ranks $\bm{r} = (p,2, \dots, 2)$}\;
            \For{$k \gets 0$ \KwTo $K_{\alpha,j}-2$}{
                Store $y_{d+2} \gets y_{d+2}\prod_{i=1}^{\alpha_{j,k}} y_{d+2+i}$ \tcc*{computes $\left(x_j^{q^k}\right)^{\alpha_{j,k}}$, $1+\alpha_{j,k}$ non-zero parameters, TT ranks $\bm{r} = (p, 2, \dots, 2)$}\;
                Store $y_{d+3},\dots,y_{d+2+q} \gets \prod_{i=1}^q y_{d+2+i}$ \tcc*{computes $x_j^{q^{k+1}}$, $1+q$ non-zero parameters, TT ranks $\bm{r}=(p, 2, \dots, 2)$}\;
            }
            Store $y_{d+2} \gets y_{d+2} \prod_{i=1}^{\alpha_{j,k}}y_{d+2+i}$ \tcc*{1 layer, $1+\alpha_{j,k}$ non-zero parameters, TT ranks $\bm{r} = (p, 2, \dots, 2)$}\;
        }
        Add monomial to polynomial $y_1 \gets y_1 + y_{d+2}$ \tcc*{1 layer, 1 non-zero parameter, TT ranks $\bm{r}= (p, 1, \dots, 1)$}\;
        \caption{Computes the $\alpha$-th monomial}
    \end{algorithm}
    Therefore, for constructing the polynomial $P$ we need
    \begin{equation*}
        \sum_{\alpha \in \Lambda} \left(1 + \sum_{j=1}^d (1+K_{\alpha,j}-1+1) + 1\right) \leq |\Lambda|\left(2 + d + d\lfloor \log_q(N_{\max}+1)\rfloor\right)
    \end{equation*}
    layers and the TT ranks are bounded by $\bm{r}=(p,2, \dots, 2)$. The reason for a $2$ to appear in $\bm{r}$ is that in the definition of CTTs, a layer is of the form $\Id + \psi$, so we have to remove the previous value of the state $y$ and replace it with the new one.
\end{proof}
\begin{remark}[Special case $q=1$]
    If we allow $q$ to be equal to $1$, then we can consider the base-$2$ representation of $\alpha_j$ since we already have $x_j$ in the state $y_{1+j}$, and one copy in the workspace $y_{d+3}$. But in order to compute the power $x_j^{\alpha_j}$, we have to evaluate $\alpha_j$ multiplications. Consequently, the number of layers is
    \begin{equation*}
        L = |\Lambda|(2+d) + \sum_{\alpha \in \Lambda}\sum_{j=1}^d \alpha_j =\mathcal{O}\left(d|\Lambda|N_{\max}\right),
    \end{equation*}
    which is linear in $N_{\max}$ instead of $\log_2(N_{\max})$ when using fast exponentiation.
\end{remark}
\begin{remark}[Encoding with a number of layers in $\mathcal{O}(|\Lambda|)$]
    If we allow the width to depend on the degrees $N_j$~\eqref{eq:Nj}, then it is possible to encode $P$ with complexity $\mathcal{O}(|\Lambda|)$. Indeed, by letting $p = 1+\sum_{j=1}^d N_j$, we store $N_j$ copies of $x_j$, so we can directly compute $x_j^{\alpha_j}$ for each $\alpha \in \Lambda$ and add the monomial $c_\alpha \prod_{j=1}^d x_j^{\alpha_j}$ to the accumulator. However, encoding polynomial approximation tools with adaptive polynomial degree would require CTTs with adaptive width. 
\end{remark}

\subsubsection{Neural networks}

By the preceding results, the CTT architecture is able to represent coordinate transformations in early layers, which may highly reduce the required rank of later layers. An immediate consequence of the ability to represent linear transformations, is that the CTT format can also represent certain types of neural networks, provided the basis is chosen correctly.

\begin{proposition}[Concatenation of two CTs]\label{proposition:ct-concatenation}
    Let $f$ and $g$ be two compositional tensors with $L_f$ (resp. $L_g$) layers, width $p_f$ (resp. $p_g$), $m_f$ (resp. $m_g$) non-zero parameters, lift $\mathfrak{L}_f$ (resp. $\mathfrak{L}_g$), retraction $\mathfrak{R}_f$ (resp. $\mathfrak{R}_g$), and common basis $\Phi$ such that $\{1\} \subset \Phi$. Then, there exists a compositional tensor $h$ with $L := \max(L_f, L_g)$ layers, width $p:=p_f+p_g$, $m:=m_f+m_g$ non-zero parameters, lift $\mathfrak{L}(x) := (\mathfrak{L}_f(x), \mathfrak{L}_g(x))$, retraction $\mathfrak{R}(y) := (\mathfrak{R}_f(y^{(f)}), \mathfrak{R}_g(y^{(g)}))$, where $y=(y^{(f)}, y^{(g)})$, and such that $h(x) = (f(x), g(x))$. 
    The TT ranks $\bm{r^\ell}$ of the tensor $\bm{\psi}_\ell$ in layer $\ell$ are given by
    \begin{equation}\label{eq:tt-ranks-concatenation}
        \begin{aligned}
            &r_0^\ell = r_0^{\ell,f} + r_0^{\ell,g},\\
            &r_{k}^{\ell} = r_{k}^{\ell,f} + 1, \quad 1 \leq k \leq p_f-1,\\
            &r_{p_f}^\ell = 1 + r_{1}^{\ell,g},\\
            &r_{p_f+t+1}^{\ell} = 1 + r_{t+1}^{\ell,g}, \quad 1 \leq t \leq p_g-2,\\
            &r_{p_f+p_g}^\ell = 1.
        \end{aligned}
    \end{equation}
\end{proposition}
\begin{proof}
    The encoding of the concatenation $h(x)=(f(x), g(x))$ is done in an explicit block-wise way. The $f$-layers (resp. $g$-layers) act only the first (resp. second) block. Because these blocks are disjoint, a single layer can simultaneously perform one $f$-step and one $g$-step. If one CT has fewer layers, we pad it with identity layers. Thus, we require only $\{1\} \subset \Phi$ and $\max(L_f, L_g)$ layers. Without loss of generality, we assume that $\phi_1(x)=1$. The state of $f$ (resp. $g$) are denoted $y^{(f)}$ (resp. $y^{(g)}$), and we set $y := (y^{(f)}, y^{(g)})$.

    We start by re-indexing the original layer tensors so they are defined for all $\ell=1 \dots L$ by padding with zero-layers. For $1 \leq \ell \leq L_f$ (resp. $1 \leq \ell \leq L_g$) let $\psi_\ell^{(f)}$ (resp. $\psi_\ell^{(g)}$) be the given $f$-layer (resp. $g$-layer), and for $\ell > L_f$ (resp. $\ell > L_g$) define $\psi_\ell^{(f)} = 0$ (resp. $\psi_\ell^{(g)} = 0$).

    Now we give the explicit encoding of the layers $\psi_\ell$. If $1 \leq j \leq p_f$ (an $f$-output coordinate), set
    \begin{equation*}
        \bm{\psi}_\ell(j, i_1, \dots ,i_{p_f}, i_{p_{f+1}}, \dots, i_p) = \begin{cases}
            \bm{\psi}_\ell^{(f)}(j, i_1, \dots, i_{p_f}) &\text{if}~ i_{p_{f+1}}=\dots=i_p=1,\\
            0 & \text{otherwise},
        \end{cases}
    \end{equation*}
    and if $p_f < j \leq p$ (an $g$-output coordinate), set
    \begin{equation*}
        \bm{\psi}_\ell(j, i_1, \dots ,i_{p_f}, i_{p_{f+1}}, \dots, i_p) = \begin{cases}
            \bm{\psi}_\ell^{(g)}(j-p_f, i_{p_f+1}, \dots, i_p) &\text{if}~ i_{1}=\dots=i_{p_f}=1,\\
            0 & \text{otherwise}.
        \end{cases}
    \end{equation*}

    Let $\bm{r^{\ell,f}}$ (resp. $\bm{r^{\ell,g}}$) be the TT ranks of $\bm{\psi}_\ell^{(f)}$ (resp. $\bm{\psi}_\ell^{(g)}$), with $r_0^{\ell,f}=p_f$ (resp. $r_0^{\ell,g}=p_g$) and $r_{p_f}^{\ell,f}=1$ (resp. $r_{p_g}^{\ell,g}=1$). Then the TT ranks $\bm{r^\ell}$ of $\bm{\psi}_\ell$ are given by \eqref{eq:tt-ranks-concatenation}. 
    In fact, $\bm{\psi}_\ell$ can be written as the sum of two tensors $A_\ell$ and $B_\ell$. Let the TT cores of $\bm{\psi}_\ell^{(f)}$ (resp. $\bm{\psi}_\ell^{(g)}$) be $F_\ell^{(k)} \in \R^{r_{k-1}^{\ell,f} \times n_k \times r_{k}^{\ell,f}}$ (resp. $G_\ell^{(k)} \in \R^{r_{k-1}^{\ell,g} \times n_k \times r_{k}^{\ell,g}}$), with $1 \leq k \leq p_f$ (resp. $p_g$). Moreover, let $E_1 \in \R^{1 \times n \times 1}$ be such that $E_1(1,i,1)=\delta_{i,1}$.
    Observe that
    \begin{itemize}
        \item $A_\ell$ is equal to $\bm{\psi}_\ell^{(f)}$ when $i_{p_f+1}=\dots=i_p=1$ and $0$ otherwise, and can be written as
        \begin{equation*}
            \begin{aligned}
            A_\ell(j,i_1,\dots,i_{p_f},i_{p_f+1},\dots,i_p) =& F_\ell^{(1)}(j,i_1,:)F_\ell^{(2)}(:,i_2,:)\dots F_\ell^{(p_f)}(:,i_{p_f},1)\\
            &\cdot \underbrace{E_1(1,i_{p_f+1},1)\dots E_1(1,i_p,1)}_{p_g}.
            \end{aligned}
        \end{equation*}
        \item $B_\ell$ can be written with $p_f$ rank-1 cores for the condition $i_1=\dots=i_{p_f}=1$ and symmetrically
        \begin{equation*}
            \begin{aligned}
                B_\ell(j,i_1,\dots,i_{p_f},i_{p_f+1},\dots,i_p) =& \underbrace{E_1(1,i_{1},1)\dots E_1(1,i_{p_f},1)}_{p_f}G_\ell^{(1)}(j,i_{p_f+1},:)\\
                &\cdot G_\ell^{(2)}(:,i_{p_f+2},:)\dots G_\ell^{(p_g)}(:,i_{p},1).
            \end{aligned}
        \end{equation*}
    \end{itemize}

    Finally, by letting $\mathfrak{L}(x) := (\mathfrak{L}_f(x), \mathfrak{L}_g(x))$ and $\mathfrak{R}(y) := (\mathfrak{R}_f(y^{(f)}), \mathfrak{R}_g(y^{(g)}))$, we have $h(x) := \mathfrak{R} \circ (\Id + \psi_L) \circ \dots \circ (\Id + \psi_1) \circ \mathfrak{L}(x)$ and this compositional tensor gives exactly the concatenation of $f$ and $g$.
\end{proof}

\begin{proposition}[Vectorization of activation function]\label{proposition:vectorization}
    Let $\sigma : \R \to \R$ be an activation function representable by a compositional tensor with $L_\sigma$ layers, width $p_\sigma$, $m_\sigma$ non-zero parameters, lift $\mathfrak{L}_\sigma$, retraction $\mathfrak{R}_\sigma(y)=y_1$, and with basis $\Phi$ such that $\{1\} \subseteq \Phi$. Then the vectorization $\bm{\sigma}(x_1,\dots,x_d)=(\sigma(x_1),\dots,\sigma(x_d))$ can be represented by a compositional tensor with $L_\sigma$ layers, width $p := d \cdot p_\sigma$, $m := d \cdot m_\sigma$ non-zero parameters, lift $\mathfrak{L}(x) := (\mathfrak{L}_\sigma(x_1), \dots, \mathfrak{L}_\sigma(x_d)) \in \R^p$, retraction $\mathfrak{R}(y) := (y_1, y_{1+p_\sigma},\dots,y_{1+(d-1)p_\sigma}) \in \R^d$.
    The TT ranks $\bm{r}^\ell$ of the tensor $\bm{\psi}_\ell$ in layer $\ell$ are given by
    \begin{equation*}
        \begin{aligned}
            r_0^\ell &= \sum_{k=1}^d r_0^{\ell,\sigma} = d \cdot p_\sigma,\\
            r_K^\ell &= r_j^{\ell,\sigma} + (d-1), \quad K=(i-1) + j,\\
            r_{dp_\sigma}^\ell &= 1,
        \end{aligned}
    \end{equation*}
    where $1 \leq i \leq d$ is in the index of the block, and $1 \leq j \leq p_\sigma - 1$ is the position within the block.
\end{proposition}
\begin{proof}
    It is a direct consequence of \cref{proposition:ct-concatenation}. We apply $\sigma$ independently to each coordinate $y_{1 + kp_\sigma}$, $0 \leq k \leq d-1$. Each input coordinate $x_i$ is lifted by $\mathfrak{L}_\sigma$, and then the $L_\sigma$ layers of $\sigma$ are applied to each block corresponding to the lifted input coordinate $x_i$. The retraction $\mathfrak{R}$ gets the output of the composition of tensors of each block.

    Similarly as in \cref{proposition:ct-concatenation}, the tensor $\bm{\psi}_\ell$ in the $\ell$-th layer can be written as a sum of $d$ tensors $U_\ell^{(k)}$, $1 \leq k \leq d$, where $U_\ell^{(k)}$ acts on block $k$ and is padded by rank-$1$ constant cores on all other blocks. This yields the desired result.
\end{proof}

\begin{theorem}[Encoding of a DNN]\label{theorem:encoding-dnn}
    Let $\sigma : \R \to \R$ be an activation function representable by a compositional tensor with $L_\sigma$ layers, width $p_\sigma$, $m_\sigma$ non-zero parameters, lift $\mathfrak{L}_\sigma(x) = (x,0,\dots,0)$, retraction $\mathfrak{R}_\sigma(y)=y_1$, and with basis $\Phi$ such that $\{1, x\} \subseteq \Phi$. Let $f = T_L \circ \sigma \circ T_{L-1} \circ \sigma \circ \dots \circ \sigma \circ T_1$ be a deep neural network of width $p_f \geq d,d_o$ with affine transformations $T_1 : \R^d \to \R^{p_f}$, $T_\ell : \R^{p_f} \to \R^{p_f}$ for $2 \leq \ell \leq L-1$, and $T_L : \R^{p_f} \to \R^{d_o}$. Then, $f$ can be exactly encoded by a compositional tensor with $L + (L-1)L_\sigma$ layers, width $p := p_f \cdot p_\sigma$, lift $\mathfrak{L}$ and retraction $\mathfrak{R}$ given by
    \begin{align*}
        \mathfrak{L}(x)_{i,k} &= \begin{cases}
            x_i & \text{if}~ 1 \leq i \leq d_f~\text{and}~k=1,\\
            0 & \text{otherwise}
        \end{cases} \in \R^{p_f p_\sigma},\\
        \mathfrak{R}(y) &= (y_{1,1}, \dots, y_{d_o,1}) \in \R^{d_o}.
    \end{align*}
    Moreover, the number of non-zero parameters is given by
    \begin{equation*}
        \sum_{\ell=1}^{L-1} (\|A_\ell - I\|_0 + \|b_\ell\|_0) + (\|A_L-I\|_0 + \|b_L\|_0) + (L-1)p_fm_\sigma.
    \end{equation*}
    Suppose that the $L_\sigma$ of $\sigma$ have TT ranks bounded by $\bm{r_\sigma}$. Then the TT ranks of each affine layer are bounded by $(p, p_f, \dots, p_f)$, and those for computing $\sigma$ are bounded by $\bm{r_\sigma}$.

    If $\sigma \in \Span{\Phi}$, then $p_\sigma=1$, $m_\sigma=1$, $L_\sigma=1$ and $\mathfrak{L}_\sigma(x)=x$, which results in an encoding of $f$ by a compositional tensor with $2L-1$ layers, with width $p_f$, and the TT ranks are bounded by $(p, p_f, \dots, p_f)$.
\end{theorem}
\begin{proof}
    We represent each neuron $j$, $1 \leq j \leq p_f$, by a block of $p_\sigma$ coordinates in the CT state. This allows us to apply the CT representation of $\sigma$ to all neurons in parallel. Consequently, the total width required for the encoding of $f$ in a CT is $p := p_\sigma p_f$. Within each block, the first coordinate stores the current scalar value of the neuron, while the remaining coordinates are used for computing $\sigma$ through its CT representation.
    Define the lift $\mathfrak{L} :\R^d \to \R^{p_fp_\sigma} $ by $\mathfrak{L}(x)_{i,j} = x_i$ if $1 \le i \le d_f$ and $j=1$ and $\mathfrak{L}(x)_{i,j} = 0$ otherwise. This choice of lift is motivated by the fact that $\mathfrak{L}_\sigma(x)=(x,0,\dots,0)$ and by \cref{proposition:vectorization}. The state is $y = (y_k)_{k=1}^p$ and we reindex it as $y=(y_{1,1},\dots, y_{1,p_\sigma},y_{2,1}, \dots, y_{p_f,p_\sigma})$ where the first index denotes the block from $1$ to $p_f$ and the second index corresponds to the local coordinate, from $1$ to $p_\sigma$.

    \paragraph{Encoding of the affine layers.}
    The encoding of the affine layers $T_\ell$ can be done similarly as in \cref{proposition:linear-transformation}. If $p_\sigma=1$, we may directly apply \cref{proposition:linear-transformation}. If $p_\sigma > 1$, the affine layer $T_\ell$ must act only on the first coordinate $y_{i,1}$ of each block. The construction is given for $2 \leq \ell \leq L-1$; the cases $\ell=1$ and $\ell=L$ are handled analogously. Since $p_\sigma > 1$, the matrix $A_\ell$ and the vector $b_\ell$ must be modified such that they only act on the variables $y_{i,1}$, $1\le i \le p_f$. Define the matrix $\widetilde{A_\ell} \in \R^{p \times p_f}$ and vector $\widetilde{b_\ell} \in \R^p$ such that 
    \begin{equation*}
        (\widetilde{A_\ell})_{k,j} := [(A_\ell)_{s,j} - \delta_{s,j}] 
        \delta_{k,{(s-1)p_\sigma+1}}, \quad (\widetilde{b_\ell})_{k} := (b_\ell)_s \delta_{k,{(s-1)p_\sigma+1}},
    \end{equation*}
    for $1 \leq s,j \leq p_f$ and $(s-1)p_\sigma +1 \leq k \leq sp_\sigma$.
    Thus, only rows with indices $k$ such that $k \equiv 1 \pmod {p_\sigma} $ are non-zero.
    Following \cref{proposition:linear-transformation} define
    \begin{equation*}
        \begin{aligned}
            U_1^\ell(y_1) &=   \begin{pmatrix}
                y_1 (\widetilde{A_\ell})_{:,1}  + \widetilde{b_\ell} &(\widetilde{A_\ell})_{:,2}   &(\widetilde{A_\ell})_{:,p_f}
            \end{pmatrix} \in \R^{p \times p_f},\\
            U_k^\ell(y_k) &= I_{p_f} + \mathbf{1}_{k \equiv 1 \pmod {p_\sigma}}
            (y_k-1){e_s e_s^\top} \in \R^{p_f \times p_f}, \quad 2 \leq k = (s-1)p_\sigma +1 \leq p-1,\\
            U_p^\ell(y_p) &= (1, \dots, 1)^\top \in \R^{p_f \times 1},
        \end{aligned}
    \end{equation*}
    where $e_s \in \R^{p_f}.$
    Using these cores to define the tensor $\psi_\ell$ yields $\psi_\ell(y)_k =T_\ell(y_{1,1}, \dots, y_{p_f,1})_s - y_{s,1}$ for $k = (s-1)p_\sigma + 1$, $1 \le s \le p_f$, and $\psi_\ell(y)_k= 0 $ otherwise.  Therefore 
    $(\Id+\psi_\ell)(y)_k = T_\ell(y_{1,1}, \dots, y_{p_f,1})_s $ for components  $k = (s-1)p_\sigma + 1$, $1 \le s \le p_f$, and $y_k$ otherwise. 
     Thus, the TT ranks for encoding an affine layer is bounded by $(p,p_f,\dots,p_f)$.
    
    \paragraph{Application of vectorized $\sigma$.}
    We want to apply $\sigma$ independently to each coordinate $y_{i,1}$. As $\mathfrak{L}_\sigma(x)=(x,0,\dots,0)$, we can directly apply the $L_\sigma$ layers $\psi^{(\sigma,1)}, \dots, \psi^{(\sigma, L_\sigma)}$ of $\sigma$, as from the previous paragraph we have the output of the affine layer on the coordinates $(y_{1,1},\dots,y_{p_f,1})$. The retraction $\mathfrak{R}_\sigma$ of $\sigma$ can be absorbed by the layers.
    Hence, the number of non-zero parameters for encoding an affine layer is
    \begin{equation*}
        \|A_\ell - I\|_0 + \|b_\ell\|_0
    \end{equation*}
    and the TT ranks are bounded by those of $\sigma$.
    In total, the construction uses $L+(L-1)L_\sigma$ layers, and the total number of non-zero parameters is
    \begin{equation*}
        \sum_{\ell=1}^{L-1} (\|A_\ell - I\|_0 + \|b_\ell\|_0) + (\|A_L-I\|_0 + \|b_L\|_0) + (L-1)p_fm_\sigma.
    \end{equation*}
\end{proof}

\begin{remark}[CTT can exactly represent a deep ReLU neural network]\label{remark:tt-nn}
    Consider feature functions $\phi_1(x) = 1$, $\phi_2(x)=x$, $\phi_3(x)=|x|$. The ReLU activation function $\sigma(x) = \max(x,0) = \frac{x+|x|}{2} = \frac 1 2 \phi_2(x) + \frac 1 2 \phi_3(x)$. By \cref{cor:linear_trafo_lift}, we can represent the affine transformation $\bm A \cdot x + b$, $\bm A\in\R^{p\times p}$, $b\in\R^p$, by a single CTT layer of rank at most $p$. Therefore, a deep ReLU neural network with $L$ layers and constant width $p$ can be written as a CTT with $2L$ layers.
\end{remark}




\subsection{Universal approximation}

To further motivate the composition of tensors and illustrate their potential, we show that they are universal approximators in $L_\mu^p(\mathcal X; \R)$ and $C(\mathcal X; \R)$ for a large class of bases $\Phi$. The encoding of multivariate polynomials give the universality in these spaces.
A rich body of results on the universality of deep neural networks has been obtained, cf.~\cite{kim2024minimum,Cybenko1989,Lu2017} to name but a few. It has also been shown in \cref{theorem:encoding-dnn} that DNNs can be exactly represented by CTTs, leading to a connection between their approximation classes.

We give here a universality result for the approximation with compositional tensors and CTTs for basis $\Phi$ that can represent the constants and the identity function $\Id$.



\begin{corollary}[Universality in $C(\mathcal{X})$]\label{corollary:universality-c}
    Assume that $1, \Id \in \Span \Phi$. Let $\mathcal{X} \subset \R^d$ be a compact domain, and consider width $p := d+1$. Then the set of compositional tensors $\mathcal{CT}(\Phi;\mathfrak{L},\mathfrak{R})$ and of compositional tensor trains $\mathcal{CTT}_{\bm{r}}(\Phi;\mathfrak{L}, \mathfrak{R})$ with $\mathfrak{L}(x):=(0,x)$, $\mathfrak{R}(y)=y_1$ and $\bm{r} \geq  (d+1,2,\dots,2)$ are dense in $C(\mathcal{X})$ equipped with the sup norm $\|\cdot\|_\infty$. The lower bound on $\bm{r}$ is tight.
\end{corollary}
\begin{proof}
    \cref{proposition:encoding-multivariate-polynomial} shows that with this choice of basis $\Phi$, width $p=d+1$, lift $\mathfrak{L}$ and $\mathfrak{R}$, we can represent any multivariate polynomial. As a direct consequence of the Stone-Weierstrass theorem, if we let the degree of the polynomial go to infinity (and so the number of layers go to infinity), we have that $\mathcal{CT}(\Phi;\mathfrak{L},\mathfrak{R})$ is dense in $C(\mathcal{X})$ equipped with the sup norm.

    Moreover, \cref{proposition:encoding-multivariate-polynomial} shows that the layers can be exactly represented by TTs with ranks equal to $\bm{r} = (d+1, 2, \dots, 2)$, which gives the density of $\mathcal{CTT}_{\bm{r}}(\Phi;\mathfrak{L}, \mathfrak{R})$ in $C(\mathcal{X})$ for the sup norm.
\end{proof}

\begin{remark}
    A direct consequence of \cref{corollary:universality-c} is that, since $C(\mathcal{X})$ is dense in $L_\mu^q(\mathcal{X})$, with $1 \leq q < \infty$ and $\mu$ a finite measure, we have the density of $\mathcal{CT}(\Phi;\mathfrak{L},\mathfrak{R})$ and $\mathcal{CTT}_{\bm{r}}(\Phi;\mathfrak{L}, \mathfrak{R})$ in $L_\mu^q(\mathcal{X})$.
\end{remark}

In an analogous manner to the approach of \cite{park2021minimum,Leshno1993}, the universality in $L_\mu^q(\R^d)$, $1 \leq q < \infty$, for a finite measure $\mu$ is shown employing the density argument of the space of compactly supported continuous functions $C_c(\R^d)$ in $L_\mu^q(\R^d)$ for the norm $\|\cdot\|_{L_\mu^q}$.

\begin{corollary}[Universality in $L_\mu^q(\R^d)$]\label{corollary:universality-lq}
    Let $\mu$ be a regular Borel finite measure, $1 \leq q < \infty$, and $1, \Id \in \Span{\Phi}$. Let the width be $p := d+1$. Then the set of compositional tensors $\mathcal{CT}(\Phi;\mathfrak{L},\mathfrak{R})$ and of compositional tensor trains $\mathcal{CTT}_{\bm{r}}(\Phi;\mathfrak{L}, \mathfrak{R})$, with $\mathfrak{L}(x):=(0,x)$, $\mathfrak{R}(y)=y_1$ and $\bm{r}=(d+1,2,\dots,2)$, are dense in $L_\mu^q(\R^d)$ for the norm $\|\cdot\|_{L_\mu^q}$.
\end{corollary}
\begin{proof}
    The proof is based on the fact that the space of continuous functions with compact support $C_c(\R^d)$ is dense in $L_\mu^q(\R^d)$, and that by \cref{corollary:universality-c} we can approximate any continuous function on a compact subset.

    Let $f^* \in L_\mu^q(\R^d)$ and $\varepsilon > 0$. Because $\mu$ is a finite measure, we can choose a radius $R > 0$ such that
    \begin{equation*}
        \|f^* - f_R\|_{L_\mu^q} \leq \frac{\varepsilon}{3},
    \end{equation*}
    where $f_R := f^* \cdot \bm{1}_{B(0,R)}$ and $B(0,R)$ is the ball of radius $R$.
    Since $\mu$ is a regular Borel finite measure, there exists a continuous $g \in C_c(\R^d)$ with support in some compact $K \subset B(0,R)$ \cite[Chapter~7]{Folland1999-qq} with
    \begin{equation*}
        \|f_R - g\|_{L_\mu^q} \leq \frac{\varepsilon}{3}.
    \end{equation*}
    The density of $\mathcal{CT}(\Phi;\mathfrak{L},\mathfrak{R})$ (resp. $\mathcal{CTT}_{\bm{r}}(\Phi;\mathfrak{L}, \mathfrak{R})$) in $C(K)$ for $\|\cdot\|_\infty$ given by \cref{corollary:universality-c} gives the existence of a function $h \in \mathcal{CT}(\Phi;\mathfrak{L},\mathfrak{R})$ (resp. $h \in \mathcal{CTT}_{\bm{r}}(\Phi;\mathfrak{L}, \mathfrak{R})$) with
    \begin{equation*}
        \|g-h\|_\infty \leq \frac{\varepsilon}{3\mu(K)^{1/q}}.
    \end{equation*}
    Then $\|g-h\|_{L_\mu^q} \leq \mu(K)^{1/q} \|g-h\|_\infty \leq \frac{\varepsilon}{3}$.

    Combining the three estimates with the triangle inequality gives the final result
    \begin{equation*}
        \|f^*-h\|_{L_\mu^q} \leq \|f^*-f_R\|_{L_\mu^q} + \|f_R-g\|_{L_\mu^q} + \|g-h\|_{L_\mu^q} \leq \varepsilon.
    \end{equation*}
\end{proof}


\begin{remark}[Approximation spaces]\label{remark:approximation-spaces}
    From the previous results, we see a relation between the \emph{approximation spaces} of DNNs and CTTs.
    The complexity of DNNs and CTTs is chosen to be the \emph{number of parameters}. The width of both is \emph{fixed to $p_f \geq d$}, which means only the number of layers is a free parameter. The corresponding approximation tools are denoted $\Sigma^{\text{DNN}} = (\Sigma_n^{\text{DNN}})_{n \in \N}$ and $\Sigma^{\text{CTT}} = (\Sigma_n^{\text{CTT}})_{n \in \N}$. By \cref{theorem:encoding-dnn}, the TT ranks are bounded by $p_f$, which gives $\mathcal{O}(Lp_f^3)$ parameters so there exists $c \geq 1$ such that $\Sigma_n^{\text{DNN}} \subseteq \Sigma_{cn}^{\text{CTT}}$. This inclusion gives $A_q^\alpha(X, \Sigma^{\text{DNN}}) \hookrightarrow A_q^\alpha(X, \Sigma^{\text{CTT}})$\footnote{The \emph{approximation space} associated with the approximation tool $\Sigma$ is defined by $A_q^\alpha(X, \Sigma) := \{f \in X : \|f\|_{A_q^\alpha} < \infty\}$, with $\|f\|_{A_q^\alpha} := \left[\sum_{n=1}^\infty (n^\alpha E(f, \Sigma_{n-1})_X)^q \frac{1}{n} \right]^{1/q}$ for $0 < q < \infty$ and $\|f\|_{A_\infty^\alpha} := \sup_{n \geq 1} n^\alpha E(f, \Sigma_{n-1})_X$, where $E(f, \Sigma_n)_X$ is the best approximation of $f$ by $\Sigma_n$. For more information, see \cite{DeVore1993}.}, i.e., $A_q^\alpha(X, \Sigma^{\text{DNN}})$ is \emph{continuously embedded} in $A_q^\alpha(X, \Sigma^{\text{CTT}})$. 

    The converse remains an open question.
\end{remark}

A further analysis of approximation properties of CTT is let for a future work. 

\subsection{Compression of a CTT}

It is well-known that given a tensor $u \in \mathcal{V}_{\bm{n}}$ we can find a tensor $\tilde{u} \in \mathcal{T}_{\bm{n},\bm{r}}$ in the tensor-train format
such that 
    $\|u-\tilde{u}\| \leq \varepsilon \|u\|,$ 
where $\| \cdot \|$ is any inner product norm induced by inner product norms on univariate function spaces by using the TT-SVD~\cite{Oseledets2011}.

One would extend this result to a compositional tensor-train, that is, we are looking for a method to compress a CTT $g = \mathfrak{R} \circ (\Id + \psi_L) \circ \dots \circ (\Id+\psi_1) \circ \mathfrak{L}$, with $\psi_k \in \mathcal{T}_{\bm{n},\bm{r}}$, by another CTT $\tilde{g} = \mathfrak{R} \circ (\Id + \widetilde{\psi_L}) \circ \dots \circ (\Id + \widetilde{\psi_1}) \circ \mathfrak{L}$, where the $\widetilde{\psi_k}$ are in the tensor-train format with probably smaller ranks.
Formally, given $g \in \mathcal{CTT}_{\leq \bm{r}}$, we are looking for $\tilde{g} \in \mathcal{CTT}_{\leq \widetilde{\bm{r}}}$ such that
\begin{equation}
    \|g-\tilde{g}\|_{L_\mu^2} \leq \varepsilon \|g\|_{L_\mu^2}.
    \label{eq:compressed-tt}
\end{equation}
We consider functions $g$ from $\R^d$ to $\R^{d_o}$ with linear lift and retraction maps of the form  $  \mathfrak{L}(x) = (x,0) \in \R^{p}$ and $  \mathfrak{R}(h) = (h_1,\ldots,h_{d_o}) \in \R^{d_o}$. Both maps $\mathfrak{L}$ and  $\mathfrak{R}$ are $1$-Lipschitz.
We work with locally Lipschitz univariate basis $\Phi$ and a compactly supported measure $\mu$, since in this case the basis $\Phi$ is Lipschitz continuous (with potentially a large Lipschitz constant). First, note that one has
\begin{equation*}
    \|g-\tilde{g}\|_{L_\mu^2} \leq \|\mathfrak{R}\|_{\ell^2 \to \ell^2} \|(\Id+\psi_L)\circ\dots\circ(\Id+\psi_1) - (\Id+\widetilde{\psi_L})\circ\dots\circ(\Id+\widetilde{\psi_1})\|_{L_{\mathfrak{L}_\sharp \mu}^2}.
\end{equation*}
Since $\mathfrak{R}$ is a projector, we have $\|\mathfrak{R}\|_{\ell^2 \to \ell^2}=1$. Therefore, it suffices to find a $\tilde{g}$ such that
\begin{equation*}
    \|(\Id+\psi_L)\circ\dots\circ(\Id+\psi_1) - (\Id+\widetilde{\psi_L})\circ\dots\circ(\Id+\widetilde{\psi_1})\|_{L_{\mathfrak{L}_\sharp \mu}^2} \leq \varepsilon \|(\Id+\psi_L)\circ\dots\circ(\Id+\psi_1)\|_{L_{\mathfrak{L}_\sharp \mu}^2}.
\end{equation*}

First, we show the following useful lemma.
\begin{lemma}\label{lemma:inequality-composition}
    Let $f_k,g_k : \R^p \to \R^p$ be Lipschitz functions, with $1 \leq k \leq L$, and $q \geq 1$. Let $\rho$ be a probability measure on $\R^p$ and assume that $f_k \circ \dots \circ f_1$ and $g_k \circ \dots \circ g_1$ are in $L_\rho^q$ for every $k=1,\dots ,L$. Then we have
    \begin{equation}
        \|f_L\circ\dots\circ f_1 - g_L\circ\dots\circ g_1\|_{L_\rho^q} \leq \varepsilon_L + \sum_{j=1}^{L-1} \varepsilon_j \prod_{k=j+1}^L \Lip(g_k),
        \label{eq:bound-diff-lipscihtz}
    \end{equation}
    where $\varepsilon_j := \|f_j - g_j\|_{L_{\rho_{j-1}}^q}$ and $\rho_j = [f_j \circ\dots\circ f_1]_\sharp \rho$, with $\rho_0=\rho$ by convention.
\end{lemma}
\begin{proof}
   We show the result by induction on $L$. Assume that the result holds for some $k$. Then,
    \begin{align*}
        \|f_{k+1}\circ\dots\circ f_1 - g_{k+1}\circ\dots\circ g_1\|_{L_\rho^q} &\leq \|f_{k+1}-g_{k+1}\|_{L_{\rho_k}^q} + \|g_{k+1}\circ f_k\circ\dots\circ f_1 - g_{k+1}\circ\dots\circ g_1\|_{L_\rho^q}\\
        &\leq \|f_{k+1}-g_{k+1}\|_{L_{\rho_k}^q} + \Lip(g_{k+1}) \|f_k\circ\dots\circ f_1 - g_k\circ\dots\circ g_1\|_{L_\rho^q}\\
        &\leq \varepsilon_{k+1} + \Lip(g_{k+1}) \left(\varepsilon_k + \sum_{j=1}^{k-1} \varepsilon_j \prod_{m=j+1}^k \Lip(g_m)\right)\\
        &= \varepsilon_{k+1} + \sum_{j=1}^k \varepsilon_j \prod_{m=j+1}^{k+1} \Lip(g_m),
    \end{align*}
    which concludes the proof.
\end{proof}

In the case of compositional tensors, the functions $f_k$ and $g_k$ are not globally Lipschitz, but only locally Lipschitz. We then have the corollary in the case where the input space $\Omega \subset \R^p$ is compact.
\begin{corollary}\label{corollary:inequality-composition-compact}
    Let $f_k,g_k : \Omega_k \to \Omega_{k+1}$ be locally Lipschitz function, with $1 \leq k \leq L$, $q \geq 1$ and $\Omega_{k+1} := f_k(\Omega_k) \cup g_k(\Omega_k)$ where $\Omega_0 := \Omega$. Let $\rho$ be a probability measure on $\Omega$ and assume that $f_k \circ \dots \circ f_1$ and $g_k \circ \dots \circ g_1$ are in $L_\rho^q$ for every $k=1,\dots,L$. Then the functions $f_k$ and $g_k$ are Lipschitz on $\Omega_k$ and the inequality \eqref{eq:bound-diff-lipscihtz} holds.
\end{corollary}

Based on \cref{corollary:inequality-composition-compact}, we design an algorithm that constructs a \emph{compressed TT} $\tilde{g}$ such that \eqref{eq:compressed-tt} is satisfied. 
\begin{proposition}
    Let $g = \mathfrak{R} \circ (\Id + \psi_L) \circ \dots \circ (\Id+\psi_1) \circ \mathfrak{L}$ be a CTT with $\psi_k \in \mathcal{T}_{\bm{n},\bm{r}}$ and $M := \|(\Id+\psi_L)\circ\dots\circ(\Id+\psi_1)\|_{L_{\mathfrak{L}_\sharp \mu}^2}$. By choosing $\widetilde{\psi_k} \in \mathcal{T}_{\bm{n},\bm{\widetilde{r}}}$ such that
    \begin{equation*}
        \delta_j := \|\bm{\psi_j} - \widetilde{\bm{\psi_j}}\|_F \leq \varepsilon \frac{M}{L \|\Phi\|_{L_{\rho_{j-1}}^2}} \prod_{k=j+1}^L \Lip(\widetilde{\psi_k})^{-1},
    \end{equation*}
    we get
    \begin{equation*}
        \|g-\widetilde{g}\|_{L_\mu^2} \leq \varepsilon \|g\|_{L_\mu^2},
    \end{equation*}
    where $\widetilde{g} := \mathfrak{R} \circ (\Id + \widetilde{\psi_L}) \circ \dots \circ (\Id+\widetilde{\psi_1}) \circ \mathfrak{L}$.
\end{proposition}
\begin{proof}
    Let $\delta_j := \|\bm{\psi_j} - \widetilde{\bm{\psi_j}}\|_F$. By \cref{corollary:inequality-composition-compact}, we have,
    \begin{equation*}
        \|(\Id+\psi_L)\circ\dots\circ(\Id+\psi_1) - (\Id+\widetilde{\psi_L})\circ\dots\circ(\Id+\widetilde{\psi_1})\|_{L_\rho^2} \leq \sum_{j=1}^L \varepsilon_j \prod_{k=j+1}^L \Lip(\widetilde{\psi_k})
    \end{equation*}
    with $\prod_{k=L+1}^L \Lip(\widetilde{\psi_k})^{-1}=1$ by convention.
    Using the Cauchy-Schwarz inequality, we immediately get
    \begin{equation*}
        \varepsilon_j \leq \delta_j \|\Phi\|_{L_{\rho_{j-1}}^2}.
    \end{equation*}
    By taking
    \begin{equation*}
        \delta_j := \|\bm{\psi_j} - \widetilde{\bm{\psi_j}}\|_F \leq \varepsilon \frac{M}{L \|\Phi\|_{L_{\rho_{j-1}}^2}} \prod_{k=j+1}^L \Lip(\widetilde{\psi_k})^{-1},
    \end{equation*}
    we get the result.
\end{proof}

Note that $\|\Phi\|_{L_{\rho_{j-1}}^2}^{-1}$ and $\|(\Id+\psi_L)\circ\dots\circ(\Id+\psi_1)\|_{L_{\mathfrak{L}_\sharp \mu}^2}$ can be estimated by Monte-Carlo sampling and that $\Lip(\widetilde{\psi_k})$ can be estimated using the gradient of $\widetilde{\psi_k}$, which we have access to. Starting from $j=L$, one can iteratively construct $\widetilde{\psi_j}$ using TT-rounding \cite{Oseledets2011} for instance.
\section{Optimization}
\label{sec:optimization}

In this section we aim to design a learning algorithm for the compositional tensor format. We propose to study two differents approaches: a first algorithm based on optimal control called the \emph{method of successive approximation (MSA)}, and the second is (a modification of) \emph{natural gradient descent}, which has recently gained popularity in training of DNNs and earlier in quantum physics.

\subsection{An optimal control perspective}
We focus on the regression problem 
\begin{equation}
    \min_{\Psi \in \mathcal{CT}^L} \|\eta - \Psi\|_{L_\mu^2(\mathcal X)}^2,
    \label{eq:ctt-regression}
\end{equation}
where $\mu$ is a finite measure, $\mathcal X := \mathcal X_1 \times\dots\times \mathcal X_d$ is the Cartesian product of closed intervals $\mathcal X_\nu$, and $\eta \in L_\mu^2(\mathcal X)$ is a target function. We assume in the following that the CT $\Psi$ is defined by functional tensors $\psi_1,\ldots,\psi_L$ and associated coefficient tensors  $\bm \psi_1,\ldots,\bm \psi_L$ together with an appropriate lift $\mathfrak{L}$ and retraction $\mathfrak{R}$ in the sense of \cref{def:flow_CTT}. Letting $\hat{X}_0\sim \mu$, the above is equivalent to
\begin{equation}
    \min_{\Psi \in \mathcal{CT}^L} \E_{\hat{X}_0\sim \mu} \|\eta(\hat{X}_0) - \Psi(\hat{X}_0)\|_2^2.
    \label{eq:ctt-regression_exp}
\end{equation}
Since the expectation is in general not computable, it is in practice replaced by an unbiased empirical estimation via i.i.d. realizations of the random variable $\hat{X}_0$. Still, since the model class is highly nonlinear as a composition of nonlinear functions, it makes sense to introduce regularization. In the following we describe several related regularization approaches, their connection, and interpretations, as well as strengths and weaknesses.

\paragraph{Stochastic optimal control} A first important observation is that a straightforward approach for regularization leads to a discrete-time stochastic optimal control problem. Consider for $X_0 = \mathfrak{L}(\hat{X}_0)$ the problem
\begin{equation}
    \left\{
    \begin{aligned}
    &\min_{\psi_1,\ldots,\psi_L \in \mathcal{M}_{\bm r}}  J(\psi_1,\dots,\psi_L) := \E_{\hat{X}_0\sim \mu} \left[ \sum_{k=1}^L \mathcal{L}_k(X_{k-1},\psi_k) + \|\eta(\hat{X}_0) - \mathfrak{R}(X_L)\|_2^2\right] \\
    &X_{k+1} = X_k + \psi_{k+1}(X_k), \quad k=0,\ldots,L-1,
\end{aligned}\label{eq:stoch_oc_prob}
 \right.
\end{equation}
for positive functions $\mathcal{L}_k\colon \R^p\times \mathcal{T}_{\bm r}\to [0,\infty)$. Since $\mathfrak{R}(X_L) = \Psi(X_0)$ by construction, the only difference between \eqref{eq:ctt-regression_exp} and \eqref{eq:stoch_oc_prob} lies in the introduction of the sum in the objective functional.  From the point of view of the regression problem, this can naturally be understood as a  regularization, in which each layer adds a penalty defined by $\mathcal{L}_k$ to the regression loss. \eqref{eq:stoch_oc_prob} has the form of a stochastic optimal control (SOC) problem with deterministic dynamics but random initial condition $X_0$. In this setting, the regularization function $\mathcal{L}_k$ is often called the \emph{running costs}, whereas the actual regression loss $\|\eta(\hat{X}_0) - \mathfrak{R}(X_L)\|_2^2$ is called the \emph{terminal costs} of the system. From the point of view of SOC, the TTs $\psi_1,\ldots,\psi_L$ are collectively called the control of the system. Hence, when we talk about an optimal control in the following, we mean a collection $\psi_1,\ldots,\psi_L$ solving \eqref{eq:stoch_oc_prob}.
First-order conditions for such a control in the general setting are given by the Pontryagin maximum principle.

\begin{lemma}[Pontryagin maximum principle (PMP)]\label{lemma:pmp-discrete_stoch}
    Consider for some $p\in\N$ the optimal control problem
    \begin{equation}
        \left\{
        \begin{aligned}
            \min_{u_1,\ldots,u_{L}} & \sum_{k=1}^{L} \mathcal{L}_k( X_{k-1},u_k) +  g( X_L)\\
            X_{k}&=  f(X_{k-1},u_k)
        \end{aligned}
        \right.,
        \label{eq:oc}
    \end{equation}
    where $ X_0 \sim \rho$, the control values $u_k$ are in a linear space $\mathcal{U}$,  $ f\colon \mathbb{R}^p\times \mathcal{U}\to \mathbb{R}^p$ denotes the dynamics, $\mathcal{L}_k\colon \mathbb{R}^p\times \mathcal{U} \to [0,\infty) $ is the running costs and $ g\colon \mathbb{R}^p\to [0,\infty)$ is the terminal costs.
    For any state-control pair $(( X_k)_{k=0}^L,(u_k)_{k=1}^{L})$ define the costates $( \lambda_k)_{k=0}^{L}$ via
    \begin{align}
         \lambda_k &= \partial_{ X_k} \mathcal{L}_k( X_k,u_{k+1}) +  \partial_{ X_k} f( X_k,u_{k+1})\cdot \lambda_{k+1}, \label{eq:ode-costate}\\
         \lambda_L &= \nabla  g( X_L).\label{eq:terminal-costate}
    \end{align}
    Then, an optimal control $u_k^*$ and the corresponding state and costate values $ X^*_k$, $ \lambda_k^*$ satisfy for each $k\in \{1,\ldots,L\}$ the following first order necessary condition 
    \begin{align}
        \mathbb{E}[H_k( X_{k-1}^*, u_{k}^*, \lambda^*_{k})] &\leq \E [H_k( X_{k-1}, u_{k},  \lambda_{k})]\label{eq:min-hamiltonian}
    \end{align}
     for all state-control costate tuples $(( X_k)_{k=0}^L,(u_k)_{k=1}^{L},( \lambda_k)_{k=0}^L)$, 
    where
    \begin{equation}
        H_k( X_{k-1}, u_{k},  \lambda_{k}) := \mathcal{L}_k( X_{k-1}, u_{k}) +  \langle  \lambda_{k},  f( X_{k-1}, u_{k})\rangle
        \label{eq:hamiltonian}
    \end{equation}
    is the Hamiltonian.
\end{lemma}
The proof of \cref{lemma:pmp-discrete_stoch} immediately follows using Lagrange multipliers.
Note that the costates $\lambda_k$ are  random variables as $X_0 \sim \rho$.

The Pontryagin maximum principle leads to the natural algorithm called method of successive approximation (MSA), which basically, given an initial control, tries to optimize iteratively the Hamiltonian by solving the state and costate equations, and then minimizes the Hamiltonian.

\begin{algorithm}[H]
    \KwRequire{$u_1^{(0)},\ldots,u^{(0)}_{L} \in \mathcal U$, $X_0 \sim \rho$}
    $\Var{\Var{i}} = 0$\;
    \While{\emph{cond} not satisfied}{%
        Let $ X^{(\Var{i})}_{k}= f(X^{(\Var{i})}_{k-1},u^{(\Var{i})}_{k})$,\quad $k=1,\ldots,L$ ,\quad $ X_{0}^{(\Var{i})} =  X_{0}$ \;
        Let $ \lambda^{(\Var{i})}_{k} = \partial_{ X^{(\Var{i})}_{k}} \mathcal{L}( X^{(\Var{i})}_{k},u^{(\Var{i})}_{k+1}) +  \partial_{ X^{(\Var{i})}_{k}} f( X^{(\Var{i})}_{k},u^{(\Var{i})}_{k+1})\cdot \lambda^{(\Var{i})}_{k+1}$,\quad $k=L-1,\ldots,0$, \quad $ \lambda_L^{(\Var{i})} = \nabla  g( X_L^{(\Var{i})})$,\;
        Compute $u_{k}^{(\Var{i}+1)} \in \argmin_{u \in \mathcal U} \E [H(X_{k-1}^{(\Var{i})}, u,  \lambda_{k}^{(\Var{i})})]$,\quad $k=1,\ldots,L$ 
        \;
        Update $\Var{i} = \Var{i} + 1$\;
    }
    \caption{Method of successive approximation (MSA)}
    \label{algorithm:msa_stoch_theory}
\end{algorithm}

\subsubsection{On the choice of the regularization}

The choice of $\mathcal{L}_k$ leads to different solutions with different interpretations.

\paragraph{Frobenius norm regularization.} A naive choice is to define $\mathcal{L}_k$ as the squared Frobenius norm of $\bm \psi_k$, i.e. $\mathcal{L}_k(X_{k-1}, \psi_k) = \frac{R}{2}\|\bm \psi_k\|_F^2$, where $R>0$ is some regularization parameter. This choice seems to make sense as an extension of a Tikhonov--like regularization to multiple layers. However, letting $\rho = \mathfrak{L}_\#\mu$ and choosing $L^2_{\rho}$-orthonormal basis functions, the space of coefficient tensors is isometric to $L^2_{\rho}$ and in particular we have $\|\bm \psi_k\|_F^2 = \|\psi_k\|_{L_\rho^2}^2$. This is not desirable since in layer $k$, $\psi_k$ no longer gets points from the distribution $\mu$, but instead from the push-forward measure $\rho_{k-1} = [(\Id + \psi_{k-1}) \circ \dots \circ (\Id + \psi_1)]_\sharp \rho \sim X_{k-1}$ (with $\rho_0=\rho$ by convention). Therefore, the regularization is not natural in the sense that it does not take into account the dynamics, and forces $\psi_k$ to be small in the $L_\rho^2$--sense.

\paragraph{Natural regularization.} The above issue can be rectified by choosing  $\mathcal{L}_k(X_{k-1}, \psi_k) = \frac{R}{2}\|\psi_k(X_{k-1})\|^2$, $R>0$, and consequently, $\E[\|\psi_k(X_{k-1})\|^2] = \|\psi_k\|_{L_{\rho_{k-1}}^2}^2$. The regularization then forces $\psi_k$ to be small in the push-forward sense, by taking into account the dynamics.

\paragraph{Connection of MSA to gradient descent.} The application of the MSA algorithm to \eqref{eq:stoch_oc_prob} with running costs chosen via a Frobenius norm or natural regularization has clear connections to gradient descent schemes. The connection between standard MSA and gradient descent in a neural network setting has already been discussed e.g. in \cite{li2018maximum}. There, it is shown that the MSA algorithm becomes standard gradient descent if one replaces the minimization of the Hamiltonian (``hard'' update) with a relaxation (``soft'' update).

When working on the full tensor space, which is linear, the connection is even stronger. Consider the running costs defined by $\mathcal{L}_k(X_{k-1},\psi_k) = \frac{R}{2}\|\bm \psi_k\|^2_F$. The Hamiltonian in this case is given by
\begin{equation*}
\begin{aligned}
    H_k(X_{k-1},\psi_k,\lambda_k) &= \frac{R}{2}\|\bm \psi_k\|_F^2 + \langle \lambda_k,X_{k-1}+\psi_k(X_{k-1})\rangle \\
    &= \frac{R}{2}\|\bm \psi_k\|_F^2 + \langle \lambda_k,X_{k-1}+\left(\langle \bm \psi_k(j,\cdot), \Phi(X_{k-1}) \rangle_F\right)_{j=1}^p \rangle,\\
     \partial_{\bm \psi}H_k(X_{k-1},\psi_k,\lambda_k) &=  R\bm\psi_k + \lambda_k \Phi(X_{k-1}),
\end{aligned}
\end{equation*}
where $\lambda_k \Phi(X_{k-1})$ is shorthand for a random variable with realizations in $\R^{p\times n\times \ldots \times n}$ defined by
\[
    \lambda_k \Phi(X_{k-1})(j,\cdot) = \lambda_{k,j}\Phi(X_{k-1}).
\]
Hence, the MSA yields an explicit update formula for the control in every step, which is
\begin{equation*}
    \bm \psi_k^{(\Var{i}+1)} = -\frac{1}{R}\E[\lambda^{(\Var{i})}_k \Phi(X^{(\Var{i})}_{k-1})] = 
 \bm \psi_k^{(\Var{i})} - \frac{1}{R}(  R\bm \psi_k^{(\Var{i})} + 
 \E[\lambda^{(\Var{i})}_k \Phi(X^{(\Var{i})}_{k-1})]) = \bm \psi_k^{(\Var{i})} - \frac{1}{R}\E[\partial_{\bm \psi}H_k(X^{(\Var{i})}_{k-1}, \psi_k^{(\Var{i})},\lambda^{(\Var{i})}_k)].
\end{equation*}
Hence, in our setting, the hard update corresponds to a gradient descent update with step size $1/R$ in the parameter space.

In general, when working in the functional setting, suppose that we want to solve the minimization problem
\begin{equation*}
    \left\{
    \begin{aligned}
    &\min_{\psi_1,\ldots,\psi_L}  J(\psi_1,\dots,\psi_L) := \E_{{X}_0\sim \rho} \left[ \sum_{k=1}^L \frac R 2\|\psi_k(X_{k-1})\|_2^2 + g(X_L)\right] \\
    &X_{k+1} = X_k + \psi_{k+1}(X_k), \quad k=0,\ldots,L-1,\\
    &\quad X_0\sim \rho
    \end{aligned}
 \right.
\end{equation*}
for functions $\psi_k \in L_{\rho_{k-1}}^2$.
The MSA algorithm aims to minimize the quantity $\E[H_k(X_{k-1}^{(q)}, \psi_k, \lambda_k^{(q)})]$ with respect to $\psi_k$, where
\begin{equation*}
    H_k(X_{k-1}^{(q)}, \psi_k, \lambda_k^{(q)}) := \frac R 2 \|\psi_k(X_{k-1}^{(q)})\|_2^2 + (\lambda_k^{(q)})^\top (X_{k-1}^{(q)} + \psi_k(X_{k-1}^{(q)}))
\end{equation*}
and $(X_{k}^{(q)})_k$ and $(\lambda_k^{(q)})_k$ are the states and costates obtained using the controls $(\psi_k^{(q)})_k$, respectively. The Riesz representative of the Fréchet derivative of $\E[H_k(X_{k-1}^{(q)}, \psi_k, \lambda_k^{(q)}]$ with respect to $\psi_k$ is denoted $\nabla_{\psi_k} \E[H_k(X_{k-1}^{(q)}, \psi_k, \lambda_k^{(q)}]$, and is given by
\begin{equation*}
    \nabla_{\psi_k} \E[H_k(X_{k-1}^{(q)}, \psi_k, \lambda_k^{(q)})] = R\psi_k + \lambda_k^{(q)},
\end{equation*}
where in this expression $\lambda_k^{(q)}$ has to be understood as a function.
The optimality conditions give
\begin{align*}
    \psi_k^{(q+1)} &= -\frac 1 R \lambda_k^{(q)}\\
    &= \psi_k^{(q)} - \frac 1 R (R\psi_k^{(q)} + \lambda_k^{(q)})\\
    &= \psi_k^{(q)} - \frac 1 R \nabla_{\psi_k} \E[H_k(X_{k-1}^{(q)}, \psi_k^{(q)}, \lambda_k^{(q)})],
\end{align*}
which is a (functional) gradient descent with step size $\frac 1 R$.

\paragraph{Augmented Hamiltonian.} It has been shown \cite{Kerimkulov2021,li2018maximum} that  classical MSA cannot be guaranteed to update the control in the optimal descent direction, and even diverges in general. To overcome this issue, we need to modify the Hamiltonian in such a way that the MSA converges. As proposed in \cite{Kerimkulov2021,li2018maximum}, we introduce the \emph{augmented Hamiltonian}
\begin{equation}
    \tilde{H}_\gamma(k, x, \lambda, \psi, \tilde \psi) := H_k( x, \lambda, \psi) + \frac 1 2 \gamma \|\psi(x) - \tilde \psi(x)\|_2^2, \quad \gamma \geq 0.
    \label{eq:augmented-hamiltonian}
\end{equation}
The minimization step of the Hamiltonian in the classical MSA is then replaced by a minimization step of the {augmented Hamiltonian}. The difference between $H$ and $\tilde H$ is that the augmented Hamiltonian now forces $\psi$ staying in a neighbourhood of $\tilde\psi$ with respect to some metric. For example, if $x = X_{k-1}$ and $\tilde\psi = \psi_k$, then it forces the new dynamic induced by $\psi$ to stay in a neighbourhood of the dynamic induced by $\psi_k$.

\subsubsection{Application in the case of tensors}

In our setting, $\psi$ is a tensor and we derive the optimality condition to minimize the \emph{augmented Hamiltonian} \eqref{eq:augmented-hamiltonian}.
Recall that any function $\psi \in \mathcal{V}_{\mathbf{n}}$ can be represented as
\begin{equation*}
    \psi : x \mapsto \langle \bm{\psi} , \Phi(x) \rangle_F = \left(\sum_{\alpha} \bm{\psi}(j, \alpha_1,\dots,\alpha_p) \phi_{\alpha_1}(x_1)\dots \phi_{\alpha_p}(x_p)\right)_{j=1}^{p},
\end{equation*}
so that $\psi$ is a \emph{vector-valued function} parametrized by a tensor $\bm{\psi} \in \R^{p \times n \times \dots \times n}$.
For notational convenience, we write $\tilde{H}_\gamma(k, x, \lambda, \bm{\psi}, \widetilde{\bm{\psi}}) = \tilde{H}_\gamma(k, x, \lambda, \psi, \tilde \psi)$.
Taking the gradient of \eqref{eq:augmented-hamiltonian} with respect to $\bm{\psi}$ leads to the expression
\begin{align*}
    \nabla_{\bm{\psi}} \tilde{H}_\gamma(k, x, \lambda, \bm{\psi}, \widetilde{\bm{\psi}}) &= \nabla_{\bm{\psi}}\left(\frac R 2 \|\bm{\psi} \cdot \Phi(x)\|^2 + \lambda^\top (x + \bm{\psi}\cdot \Phi(x)) + \frac 1 2 \gamma \|\bm{\psi}\cdot \Phi(x) - \widetilde{\bm{\psi}}\cdot \Phi(x)\|^2\right)\\
    &= R G(x)  \bm{\psi} + \sum_{i=1}^d \lambda_i (e_i \otimes \Phi(x)) + \gamma G(x)  (\bm{\psi} - \widetilde{\bm{\psi}}),
\end{align*}
where
\begin{equation*}
    G : \left\{
    \begin{aligned}
        \R^d & \longrightarrow \mathcal{L}(\R^p \otimes \R^n \otimes\dots\otimes \R^n)\\
        x &\longmapsto I_d \otimes (\phi^{(1)}(x_1) \phi^{(1)}(x_1)^\top) \otimes\dots\otimes (\phi^{(d)}(x_d) \phi^{(d)}(x_d)^\top)
    \end{aligned}
    \right.
\end{equation*}
gives a symmetric and positive-definite operator from tensor space $\R^{p \times n \times \ldots \times n}$ to itself at each point $x \in \R^d$.

In our case, we use $B$ samples $(x_{j,k}^{(i)})_{j=1}^b$ for estimating the expectations leading to the minimization problem
\begin{equation*}
    \bm{\psi}_{k}^{(\Var{i}+1)} \in \argmin_{\bm{\psi}} \frac{1}{B}\sum_{j=1}^B \tilde H_\gamma(k-1, x_{j,k-1}^{(\Var{i})}, \lambda_{j,k}^{(\Var{i})}, \bm{\psi}, \bm{\psi}_k^{(\Var{i})}).
\end{equation*}
Setting the gradient of the functional with respect to $\bm{\psi}$ to zero yields the linear tensor equation
\begin{equation*}
    \left(\frac{1}{B} \sum_{j=1}^B G(x_{j,k-1}^{(i)})\right)  ((R+\gamma)\bm{\psi}_k^{(i+1)} - \gamma \bm{\psi}_k^{(i)}) = -\frac{1}{B}\sum_{j=1}^B \sum_{n=1}^d \lambda_{j,k,n}^{(i)} (e_n \otimes \Phi(x_{j,k-1}^{(i)})).
\end{equation*}
Such an equation can be solved using ALS \cite{holtzAlternatingLinearScheme2012} or AMEn \cite{dolgovAlternatingMinimalEnergy2014}. Note that the operator $\frac{1}{B} \sum_{j=1}^B G(x_{j,k-1}^{(i)})$ and the right-hand side have  representations in canonical tensor format with canonical tensor rank $Bd$.
In order to have the updates $\bm{\psi} _k^{(\Var{i}+1)}$ in the Tensor-Train format, one can perform a TT-rounding, or a truncation to some given ranks $\bm{r}$ \cite{Oseledets2011}.

\subsection{Natural gradient descent}
\label{sec:ngd}

Gradient descent algorithm is a prominent optimization method, especially for neural networks, and have shown good convergence properties in many applications. It has been proven that its stochastic version (this means, with a stochastic estimation of the true gradient) converges to critical points, even for non-convex objective functions \cite{Davis2019}. Furthermore, it can escape and avoid saddle points \cite{jin2017escape,jin2021nonconvex}. However, vanilla gradient descent tends to converge slowly, and often does not achieve a high accuracy. 
Following \cite{muller2024position}, we propose to employ the \emph{natural gradient descent} \cite{amariNaturalGradientWorks1998} for this specific architecture. We first provide a formal derivation.

Let $\mathcal{M} \subset L_\mu^2 := \mathcal{H}$ be a nonlinear model class. Consider the minimization problem
\begin{equation}
    \inf_{u \in \mathcal{M}} \mathcal{L}(u), \qquad \mathcal{L} : L_\mu^2 \to \R.
    \label{eq:min-pb-func}
\end{equation}
We aim to solve this problem using an iterative scheme of the form,
\begin{equation}
    u_{k+1} = u_k + \alpha_k d_k, \qquad d_k = -\nabla\mathcal{L}(u_k).
    \label{eq:descent-func}
\end{equation}

\begin{remark}
    In general, for an arbitrary Hilbert space $\mathcal{H}$, $d_k$ is defined by $d_k = -Q_{u_k}^{-1}(\nabla\mathcal{L}(u_k))$, where $\nabla\mathcal{L} \in \mathcal{H}$ is the gradient of $\mathcal{L}$ and $Q_{u_k}$ is an invertible, symmetric and positive definite linear map. For example,
    \begin{itemize}
        \item by choosing $d_k = -\nabla \mathcal{L}(u_k)$ we recover the \emph{Hilbert space gradient descent},
        \item by choosing $d_k = -\nabla^2 \mathcal{L}(u_k)^{-1}(D\mathcal{L}(u_k))$, with $\nabla^2 \mathcal{L}(u_k): \mathcal{H}\to \mathcal{H}$ the Hessian of $\mathcal{L}$ at $u_k$,  we recover the \emph{Hilbert space Newton method}. 
    \end{itemize}
\end{remark}

Henceforth, we consider a \emph{parametric class of functions}, $\mathcal{M} = \{u_\theta = P(\theta) : \theta \in \Theta\}$ where $P : \Theta \to \mathcal{M} \subseteq \mathcal{H}$ is the \emph{parametrization} and is assumed to be $C^1$, and $\Theta \subseteq \R^p$ is called the \emph{parameter space}. Since in practice, performing a descent in the functional space is infeasible, we aim to design an iterative scheme in the parameter space that mimics the descent \eqref{eq:descent-func}. Namely,
\begin{equation}
    \theta_{k+1} = \theta_k + \alpha_k w_k,
    \label{eq:update-ngd}
\end{equation}
such that $u_{\theta_{k+1}} \approx u_{k+1}$. Applying the Taylor theorem to $u_{\theta_{k+1}} = P(\theta_k + \alpha_k w_k)$ gives
\begin{equation*}
    u_{\theta_{k+1}} = u_{\theta_k} + \alpha_k DP(\theta_k)w_k + o(\alpha_k \|w_k\|).
\end{equation*}
Therefore, we would like to find $w_k$ such that we have $DP(\theta_k) w_k \approx d_k$.
The quantity $DP(\theta) w$ lives in the \emph{tangent space} of $\mathcal{M}$ at $u_\theta$, defined as
\begin{equation*}
    T_{u_\theta}\mathcal{M} := \Span\{\partial_{\theta_1} u_\theta,\dots, \partial_{\theta_p} u_\theta\} = \Span(DP(\theta)) \subset L_\mu^2.
\end{equation*}
This tangent space is equipped with the metric $\langle \cdot, \cdot\rangle_{L_\mu^2}$.
\begin{remark}
    In general, the tangent space $T_{u_\theta}\mathcal{M}$ can be equipped with a metric $\langle Q_{u_k} \cdot, \cdot\rangle_{\mathcal{H}}$ depending on $u_k$. The Hilbert space Newton method corresponds to $ Q_{u_k} = \nabla^2 \mathcal{L}(u_k)$.
\end{remark}

It follows that the $w_k$ that ``aligns the best with the functional descent'' is given by,
\begin{equation}
    w_k \in \argmin_{w \in \Theta} \frac{1}{2} \|DP(\theta_k) w - d_k\|_{L_\mu^2}^2.
    \label{eq:parameter-opt-direction}
\end{equation}

\begin{remark}[Gauss-Newton iterate]
    The update direction \eqref{eq:parameter-opt-direction} coincides with one Gauss-Newton step for the minimization of $L(\theta) := \frac{1}{2} \|P(\theta) - u\|_{\mathcal{H}}^2$, with $\mathcal{H}=L^2_\mu$. Indeed, a Gauss-Newton iteration is given by $\theta_{k+1} = \theta_k + w_k$, with $w_k$ minimizing over $w\in \Theta$ the function 
    $$
w\mapsto \frac{1}{2} \Vert P(\theta_k) + DP(\theta_k)w - u\Vert_{\mathcal{H}}^2 =  \frac{1}{2} \Vert DP(\theta_k)w - d_k\Vert_{\mathcal{H}}^2
    $$
    with $d_k = u - P(\theta_k) = -\nabla L(\theta_k) \in \mathcal{H}$. 

\end{remark}

Computing the normal equation of \eqref{eq:parameter-opt-direction} gives the update direction 
\begin{equation}
    w_k = -G(\theta_k)^\dagger \nabla_\theta L(\theta_k),
\end{equation}
where $L = \mathcal{L} \circ P : \Theta \to \R$ and $G(\theta_k) \in \R^{p \times p}$ is the \emph{Gramian matrix} defined by,
\begin{equation}
    G(\theta)_{ij} := \langle \partial_{\theta_i} u_\theta, \partial_{\theta_j} u_\theta\rangle_{L^2_\mu},
\end{equation}
and by $G(\theta)^\dagger$ we denote a \emph{pseudo-inverse} of $G(\theta)$, that means a matrix satisfying $GG^\dagger G = G$.

As illustrated in \cref{fig:diff-euclidean-natural-grad} for a particular $2$-dimensional model class $\mathcal{M}$, the flow obtained by the natural gradient  naturally points towards the optimal point, and the streamlines are more straight. 
\begin{figure}
    \centering
    \begin{subfigure}{0.49\linewidth}
        \centering
        \includegraphics[width=\linewidth]{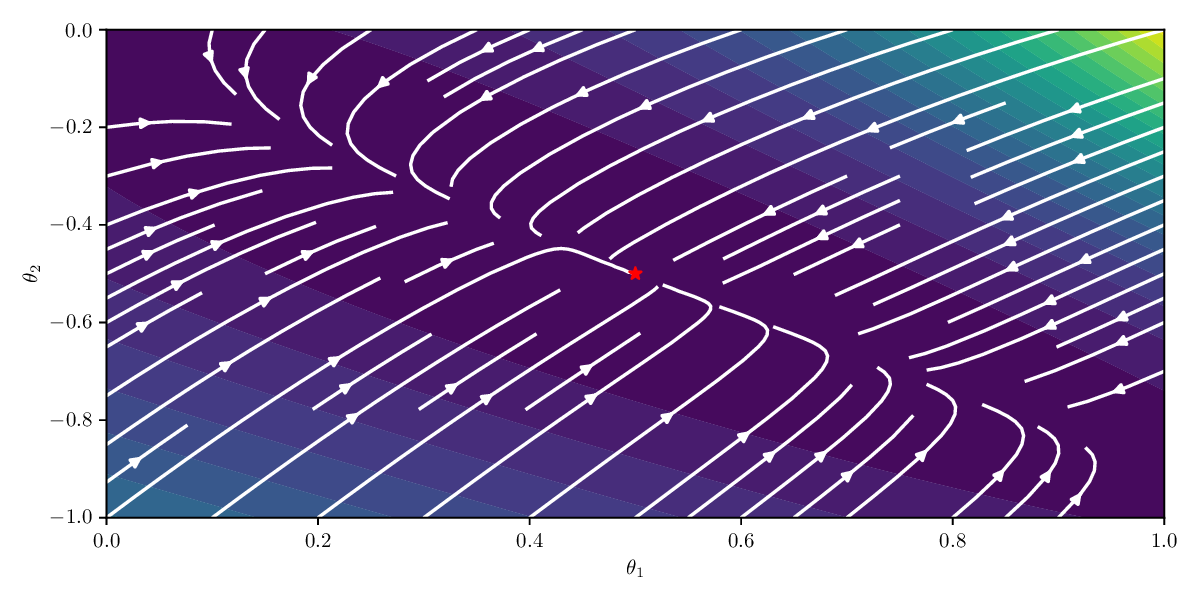}
        \caption{Euclidean gradient $\nabla_\theta L(\theta)$}
        \label{fig:euclidean-grad}
    \end{subfigure}
    \begin{subfigure}{0.49\linewidth}
        \centering
        \includegraphics[width=\linewidth]{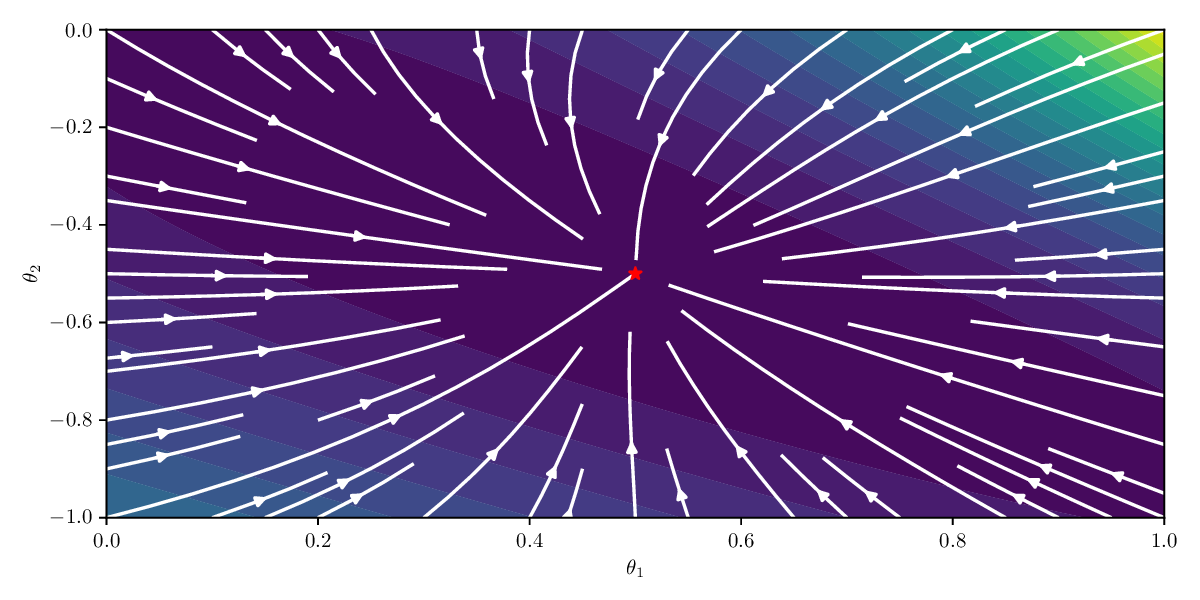}
        \caption{Natural gradient $\widetilde{\nabla}_\theta L(\theta) := G(\theta)^\dagger \nabla_\theta L(\theta)$}
        \label{fig:natural-grad}
    \end{subfigure}
    \caption{Visualization of the gradient fields $\nabla_\theta L(\theta)$ and $\widetilde{\nabla}_\theta L(\theta)$ for the model $u_\theta(x,y)=\exp(\theta_1 x + \cos(y-\theta_2))$ and the loss function $\mathcal{L}(u) = \frac{1}{2} \|u-u^*\|_{L^2([0,1]^2)}^2$.}
    \label{fig:diff-euclidean-natural-grad}
\end{figure}

An interesting property of the natural gradient is that it is reparametrization invariant \cite{van2023invariance}. That means, if we consider two parametrizations $P(\theta)$ and $ P(f^{-1}(\xi))$, where $f$ is a diffeomorphism, the update direction in $\mathcal{H}$ is the same. This implies that an update in the parameters for both parametrizations by an \emph{infinitesimal amount}, this would give the same result on the model. However, it is important to point out that updating the parameters by a \emph{finite amount} will, in general, give different locations on the model: the \emph{natural gradient descent is dependent on the choice of parametrization}.

We would also point out that it is believed that natural gradient descent can avoid saddle points because it updates the parameters in the optimal direction in $\mathcal{M}$ \cite{inoue2003line}. However, unlike the Hessian, the Gramian has only \emph{non-negative} eigenvalues, therefore it does not exploit explicitly negative curvature, which helps to escape saddle points. Consequently, it does not actively push away from saddle points. In \cite{mizutani2010analysis}, the authors argue that around a stationary solution, the Gramian matrix is rank-deficient, thefore the natural gradient descent may converge to instationary points and fail to escape.  Despite this undesired behavior, natural gradient descent has shown very good performance, particularly for PINNs \cite{muller2023achieving,nurbekyan2023efficient}. However, when employed with neural networks, the computation of the natural gradient is very resource-intensive, therefore its usage is very limited. Some methods have been developed to compute an approximation of this quantity, namely K-FAC \cite{martens2015optimizing,ba2017distributed,george2018fast} and techniques using randomized linear algebra \cite{bioli2025acceleratingnaturalgradientdescent}.

\paragraph{Implementation for CTTs.} In our case, if we want to exploit the tensor algebra, we cannot apply the natural gradient descent directly to the model class $\mathcal{CT}(\Phi)$ (and also $\mathcal{CTT}_{\leq \bm{r}}(\Phi)$), since it is not clear whether the sum of two compositional tensors yields a compositional tensor (even with probably a different number of layers). Instead, we propose to apply it layerwise, i.e. on each tensor $\psi_\ell$ on the $\ell$-th layer, since in this case we preserve the tensor structure (and also the low-rank format if used). The function $u_\theta$ is an element of $\mathcal{CT}(\Phi)$, where $\theta = (\bm{\psi_1},\dots,\bm{\psi_L})$ represents the tensors of the layers.
By applying the chain rule, we immediately get
\begin{equation}
    \frac{\partial u_\theta}{\partial \bm{\psi}_\ell}(x) = \underbrace{\frac{\partial u_\theta}{\partial u_\ell}(x)}_{\in \R^{d_o \times d}} \otimes \Phi^{\otimes d}(u_{\ell-1}(x)) \in \R^{d_o \times d} \otimes (\R^m)^{\otimes d},
    \label{eq:chain-rule-param}
\end{equation}
where $u_\ell(x) = (\Id + \psi_\ell) \circ \dots \circ (\Id+\psi_1) \circ \mathfrak{L}(x)$. The Gramian of the $\ell$-th layer is then given by
\begin{equation}
    G_\ell(\theta) = \E_{x}\left[\left(\frac{\partial u_\theta}{\partial u_\ell}(x)^\top \frac{\partial u_\theta}{\partial u_\ell}(x)\right) \otimes \left((\Phi\Phi^\top)^{\otimes d}(u_{\ell-1}(x))\right)\right] \in \R^{d\times d} \otimes (\R^{m \times m})^{\otimes d}.
    \label{eq:gram-tensor}
\end{equation}
For computing the natural gradient layerwise we thus have to solve the normal equation
\begin{equation}
    G_\ell(\theta) \cdot d = \E_x\left[\left(\frac{\partial u_\theta}{\partial u_\ell}(x)\right)^\top \nabla \mathcal{L}(u_\theta)(x)\right].
    \label{eq:ng-layerwise-normal}
\end{equation}
In our case, we can approximate the $L_\mu^2$-inner product by a Monte-Carlo approximation using $B$ samples and therefore consider an empirical version of \eqref{eq:ng-layerwise-normal}. Indeed, a Monte-Carlo approximation of $G_\ell(\theta)$ is given by
\begin{equation}
    \widehat{G_\ell(\theta)} = \frac{1}{B} \sum_{i=1}^B \underbrace{\left(\frac{\partial u_\theta}{\partial u_\ell}(x^i)^\top \frac{\partial u_\theta}{\partial u_\ell}(x^i)\right)}_{C_{\ell,0}(x^i)} \otimes \underbrace{\Phi(u_{\ell-1}(x^i)_1) \Phi(u_{\ell-1}(x^i)_1)^\top}_{C_{\ell,1}(x^i)} \otimes \dots \otimes \underbrace{\Phi(u_{\ell-1}(x^i)_d) \Phi(u_{\ell-1}(x^i)_d)^\top}_{C_{\ell,d}(x^i)}
    \label{eq:gram-tensor-mc}
\end{equation}
In general, if $u_\theta \in \mathcal{CT}(\Phi)$ and no low-rank format is assumed, the time complexity of solving \eqref{eq:ng-layerwise-normal} is in $\mathcal{O}(\min(Bd_o m^{2d}, B^2 d_o^2 m^d))$, where $B$ is the number of samples used for estimating the inner product, $d_o$ is the output dimension of the model, $d$ is the dimension of the lifted space, and $m$ is the number of univariate basis elements of $\Phi$. Moreover, the storage complexity is in $\mathcal{O}(m^d)$ for each layer, which is infeasible in high dimensions. To circumvent this problem, one can assume low-rank format for the layers $\bm{\psi_\ell}$, e.g. the tensor-train format. For solving \eqref{eq:ng-layerwise-normal} with a low-rank tensor format, one can use the \emph{alternating linear scheme (ALS)} \cite{holtzAlternatingLinearScheme2012}. However, the Gram matrix associated to the layer $\ell$ may have a bad condition number, so that ALS without preconditioning may show a   slow convergence and yield a highly suboptimal low-rank approximation of the update direction. We should also note that in this case, the update \eqref{eq:update-ngd} may leave the manifold, therefore it is required to use a \emph{retraction} $R_U : T_U \mathcal{M} \to \mathcal{M}$. The update for each layer $\bm{\psi_\ell}$ is therefore given by
\begin{equation}
    \bm{\psi}_\ell^{(k+1)} = R_{\bm{\psi}_\ell^{(k)}}(-\alpha^{(k)}\bm{w}_\ell^{(k)}),
    \label{eq:update-ngd-low-rank}
\end{equation}
where the iteration number $k$ is written at superscript for ease of reading.

\begin{algorithm}[H]
    \KwRequire{Samples $(x^i)_{i=1}^B$, initial tensors $(\bm{\psi}_\ell)_{\ell=1}^L$}
    $\Var{k} \gets 0$\;
    \While{not converged}{%
        \For{$\ell \gets 1$ \KwTo $L$}{
            Assemble the Gramian $\widehat{G_\ell(\theta)}$ \eqref{eq:gram-tensor-mc}\;
            Compute the gradient $\nabla_{\bm{\psi}_\ell} \mathcal{L}(u_\theta)(x^i)$\;
            Solve the normal equation \eqref{eq:ng-layerwise-normal} and store $\bm{w}_\ell^{(k)}$\;
        }
        Compute the learning rate $\alpha^{(k)}$\;
        \For{$\ell \gets 1$ \KwTo $L$}{
            \If{low-rank}{
                Update $\bm{\psi}_\ell^{(k+1)} \gets R_{\bm{\psi}_\ell^{(k)}}(-\alpha^{(k)} \bm{w}_\ell^{(k)})$\;
            }
            \Else{
                Update $\bm{\psi}_\ell^{(k+1)} \gets \bm{\psi}_\ell^{(k)}-\alpha^{(k)} \bm{w}_\ell^{(k)}$
            }
        }
        $\Var{k} \gets \Var{k} + 1$\;
    }
    \caption{Layerwise natural gradient for CTTs}
    \label{algorithm:layerwise-ngd}
\end{algorithm}

The Gram matrices $G_\ell$ are typically highly ill-conditioned, which poses significant challenges for solving the corresponding normal equations \eqref{eq:ng-layerwise-normal}.
A well-established approach to mitigate this issue is to approximate $G_\ell$ with a low-rank surrogate. In particular, the randomized Nystr\"om method \cite{Nystrm1930,williams2000using,pmlr-v28-gittens13,Martinsson2020} achieves this by projecting $G_\ell$ onto a randomly generated, low-dimensional subspace. In the context of this work, the Gram matrix $G_\ell$ is a linear operator acting on tensor spaces $G_\ell : \R^{d \times n \times \dots \times n} \to \R^{d \times n \times \dots \times n}$. To construct a suitable low-rank approximation, we introduce a tensor-structured sketch $S \in \R^{s \times d \times n \dots \times n}$ \cite{rudelson2012row,sun2021tensor}, where $s \in \N$ is the sketching size. The projections $G_\ell S_j \in \R^{d \times n \dots \times n}$ are computed for $1 \leq j \leq s$, and we seek a solution in the span of these random projections of $G_\ell$. We choose the sketch $S$ to be of the form
\begin{equation*}
    S_j =\frac{1}{\sqrt{s}} s_{j,
    0} \otimes \dots \otimes s_{j,d}, \quad s_{j,p} \sim \mathcal{N}(0, I).
\end{equation*}

The contraction of the Monte-Carlo approximation of $G_\ell$ given in \eqref{eq:gram-tensor-mc} and $S_j$ can be efficiently computed
\begin{equation*}
    \widehat{G_\ell} S_j = \frac{1}{B} \sum_{i=1}^B C_{\ell,0}(x^i) s_{j,0} \otimes C_{\ell,1}(x^i) s_{j,1} \otimes \dots \otimes C_{\ell,d}(x^i) s_{j,d}.
\end{equation*}
Moreover, it can be computed without forming explicitly the Jacobian by using the \emph{Jacobian-vector product (JVP)}, see remark in the subsequent numerics section.

The key advantage of this Gaussian sketching approach is that, with high probability, the span of the sketch captures the dominant eigenspace of $G_\ell$. Consequently, the low-rank approximation retains the principal components corresponding to the largest eigenvalues while attenuating the directions associated with small eigenvalues. The resulting system behaves as if a spectral filter is applied to $G_\ell$, that is, the solution lies in the subspace spanned by the most significant eigenvectors of $G_\ell$.

\section{Numerical results}\label{sec:numerics}

In this section, we would like to observe the performance of the  \emph{natural gradient descent (NGD)} and MSA algotihms presented in previous section, compared to other (higher-order) optimization methods such as Adam, L-BFGS and BFGS. All the experiments are written in the JAX python framework \cite{jax2018github} and can be found on the GitHub repository \url{https://github.com/chmda/ctt}. For assembling the normal equation \eqref{eq:ng-layerwise-normal}, the Jacobians and the right-hand side are computed using \texttt{jax.vjp} instead of simply calling \texttt{jax.jacfwd}, speeding up the computation.

\subsection{Layers with full tensors (no low-rank formats)}

We first study the behaviour of the algorithms where it is not assumed that the different layers $\psi_\ell$ are in a low-rank format, i.e., we consider full tensors. The basis functions are chosen to be $\Phi = \{1, \Id\}$. For the considered approximation and recovery problems, the \emph{training dataset} $\mathcal{D}_{\text{train}}$ consists of $100,000$ samples drawn from the measure $\mu$ (which are specified for each problem), and the \emph{validation dataset} $\mathcal{D}_{\text{val}}$ consists of $2,000$ independent samples from $\mu$. The loss function considered is 
\begin{equation*}
    \mathcal{L}(u) := \frac{1}{2}\|u-u^*\|_{L_\mu^2(\Omega)}^2, \quad \mathcal{L} : L_\mu^2(\Omega) \to \R^+,
\end{equation*}
where $\mu$ is the uniform measure on $\mathcal{X}$. This loss function is however replaced by its empirical version,
\begin{equation*}
    \widehat{\mathcal{L}(u)} := \frac{1}{2} \sum_{x \in \mathcal{D}} |u(x)-u^*(x)|^2,
\end{equation*}
where $\mathcal{D}$ is an ensemble of random points, e.g. a mini-batch or the whole training dataset. The performance of the algorithms is measured using the $L_\mu^2(\mathcal{X})$-relative error estimated by
\begin{equation*}
    \hat{\varepsilon}^2 := \frac{\sum_{x \in \mathcal{D}_{\text{val}}} |\widetilde{u}(x) - u^*(x)|^2}{\sum_{x \in \mathcal{D}_{\text{val}}} |u^*(x)|^2},
\end{equation*}
where $\widetilde{u}$ is an approximation of $u^*$. If specified, the line-search algorithm used for finding optimal $\alpha_k$ has to satisfy the \emph{strong-Wolfe conditions} \cite[Chapter~3]{nocedal2006numerical}. Since the Gram matrix $G_\ell(\theta)$ can be rank-defficient and the computation of the natural gradient $w_\ell$ by solving the least-squares problem \eqref{eq:parameter-opt-direction} can be tedious, we propose to add a \emph{Tikhonov regularization} to \eqref{eq:ng-layerwise-normal}, which translates into adding a regularization parameter $\lambda \cdot \Id$ to $G_\ell(\theta)$ so that now we have to solve the linear equation 
\begin{equation*}
    (G_\ell(\theta) + \lambda \Id) \cdot d = \E_x\left[\left(\frac{\partial u_\theta}{\partial u_\ell}(x)\right)^\top \nabla \mathcal{L}(u_\theta)(x)\right].
\end{equation*}
This method can be seen as a \emph{trust-region approach}. It shifts the spectrum of $G_\ell(\theta)$, making the solving more stable. The choice of $\lambda$ remains an open question, but a rule of thumb is to increase $\lambda$ when the algorithm tends to diverge and to decrease it otherwise. The tensor initialization is also an important key for the good behaviour of the optimization algorithms. In order to avoid the explosion of the variance with the layers and the number of features, we propose to choose the tensors as
\begin{equation*}
    \bm{\psi_\ell} \sim \mathcal{N}\left(0, \frac{c}{Ln^d}\right), \quad 1 \leq \ell \leq L,
\end{equation*}
where $L$ is the number of layers, $n$ is the number of features of $\Phi$, and $c > 0$ is a positive constant chosen to be $c=2$. This choice is motivated by the works on neural network weight initialization \cite{he2015delving,zhang2019fixup}.

\subsubsection{A recovery problem}

Let $d \geq 2$ and
\begin{align}
    u^*(x)&=\prod_{i=1}^m (x_i y - x_{m+i})^2, \quad x \in \mathcal{X} = [0,1]^d,\\
    y &= \begin{cases}
        1 & \text{if}~d=2m\\
        x_d & \text{if}~d=2m+1
    \end{cases}.
    \label{eq:recovery-2}
\end{align}
This function can be exactly represented by a compositional tensor with 2 layers. Indeed, for simplicity take $d=2m$. For the case $d=2m+1$, we just have to multiply the terms $h_i$, $1 \leq i \leq m$, by $h_{2m+1}$. Let
\begin{align}
    u_1(h) &= (\underbrace{h_1-h_m,h_2-h_{m+1},\dots,h_m-h_{2m}}_{m},\underbrace{h_1-h_m,h_2-h_{m+1},\dots,h_m-h_{2m}}_{m}),\\
    u_2(h) &= (h_1 h_2 \dots h_{2m}, 0, \dots, 0),
    \label{eq:recovery-2-encoding}
\end{align}
then we have exactly $u^* = u_2 \circ u_1$. Furthermore, $u_1$ can be represented by a TT with ranks bounded by $2$ since it is just a difference of two variables, and $u_2$ can be represented by a TT with ranks equal to 1 since it is just the product of all variables.

For representing this function by a TT with polynomial basis, it must have maximal ranks, which grow exponentially with $d$. For example, if $d=4$, then the TT ranks are $(1, 3, 9, 3)$ and the number of parameters is $180$. This result can be verified experimentally: for the ranks ranging from $1$ to $9$, the relative $L^2$ error is $\{0.84, 0.68, 0.46, 0.16, 0.14, 0.11, 0.079, 0.013, 0\}$ (averaged over $25$ experiments).

The lifting operator is chosen as $\mathfrak{L}=\Id$ and the retraction operator as $\mathfrak{R}=e_1$. The experiments are conducted for $d=4$ and $d=5$. The number of training samples is fixed to $2048$, and the number of validation samples is fixed to $512$. Regarding hyperparameters for each experiment,
\begin{itemize}
    \item for $d=4$, the learning rate is fixed to $\alpha_k=0.7$, and the regularization coefficient is $\lambda=10^{-12}$,
    \item for $d=5$, the learning rate is chosen constant $\alpha_k=0.5$, and the regularization coefficient is $\lambda=10^{-11}$.
\end{itemize}

\begin{figure}[h]
    \centering
    \begin{subfigure}{0.49\linewidth}
        \centering
        \includegraphics[width=\linewidth]{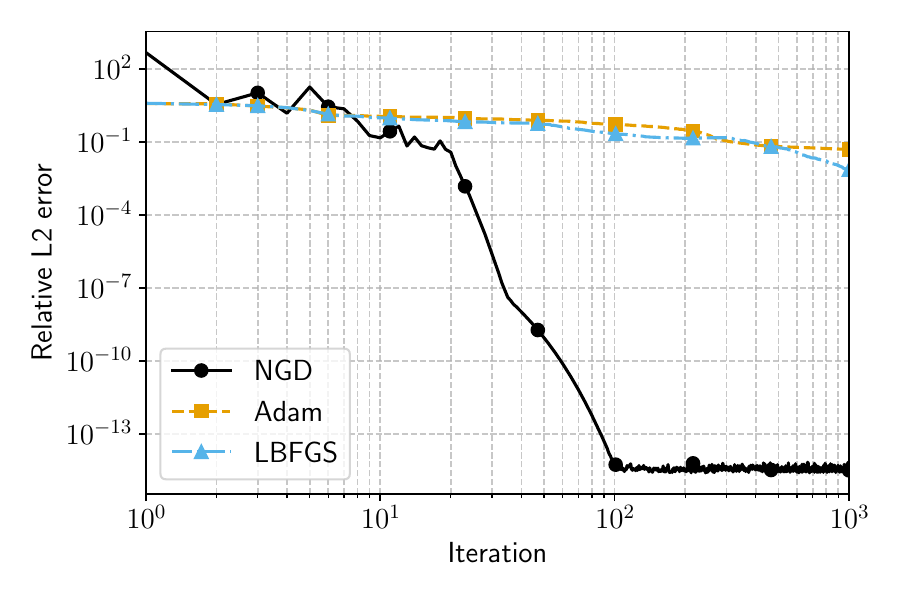}
        \caption{$d=4$}
        \label{fig:recovery-2-convergence-d4}
    \end{subfigure}
    \begin{subfigure}{0.49\linewidth}
        \centering
        \includegraphics[width=\linewidth]{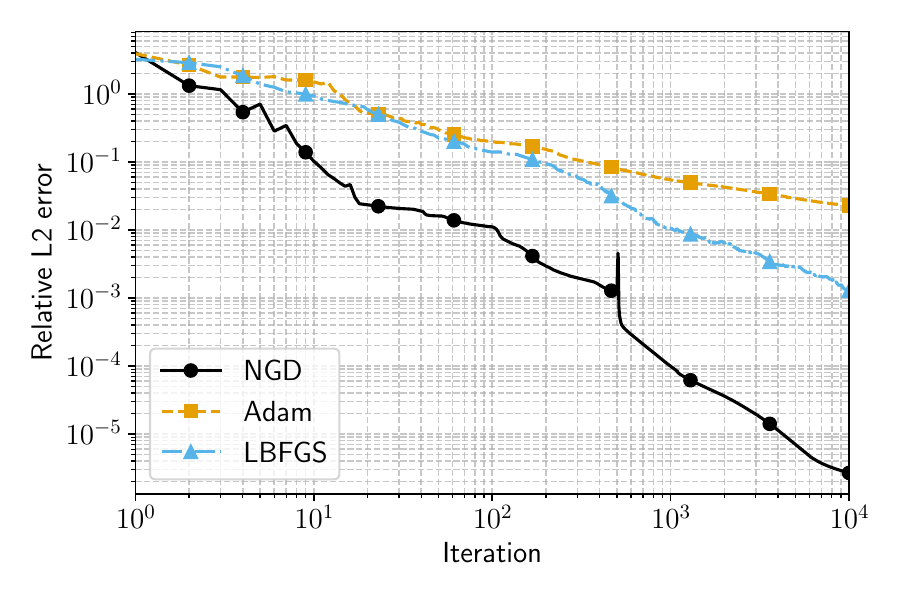}
        \caption{$d=5$}
        \label{fig:recovery-2-convergence-d5}
    \end{subfigure}
    \caption{Convergence plot for the optimizers Adam, NGD and L-BFGS for the recovery problem \eqref{eq:recovery-2} in log-log scale, for dimensions $d=4,5$.}
    \label{fig:recovery-2-convergence}
\end{figure}

The plots \cref{fig:recovery-2-convergence} show that the NGD converges faster to an exact solution (or a good approximation)   than the other optimisation algorithms. In \cref{fig:recovery-2-convergence-d4} we see that the NGD converges with a fast rate and converges quickly to the exact solution. In \cref{fig:recovery-2-convergence-d5}, the NGD does not converge to an exact solution of the problem. While it converges linearly, the rate is much better than of the other methods and a much lower error is reached after the same number of iterations.

\begin{figure}[h]
    \centering
    \begin{subfigure}{0.49\linewidth}
        \centering
        \includegraphics[width=\linewidth]{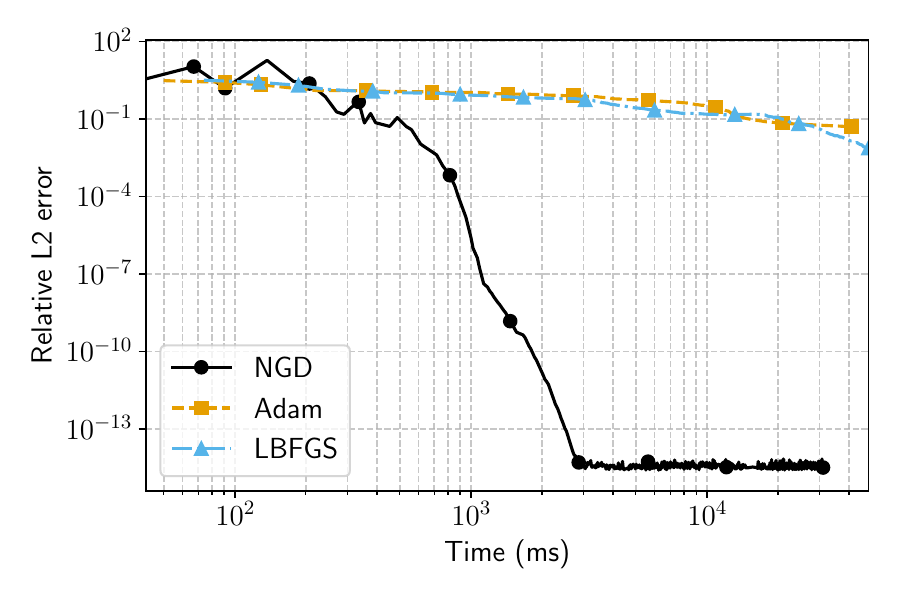}
        \caption{$d=4$}
        \label{fig:recovery-2-time-d4}
    \end{subfigure}
    \begin{subfigure}{0.49\linewidth}
        \centering
        \includegraphics[width=\linewidth]{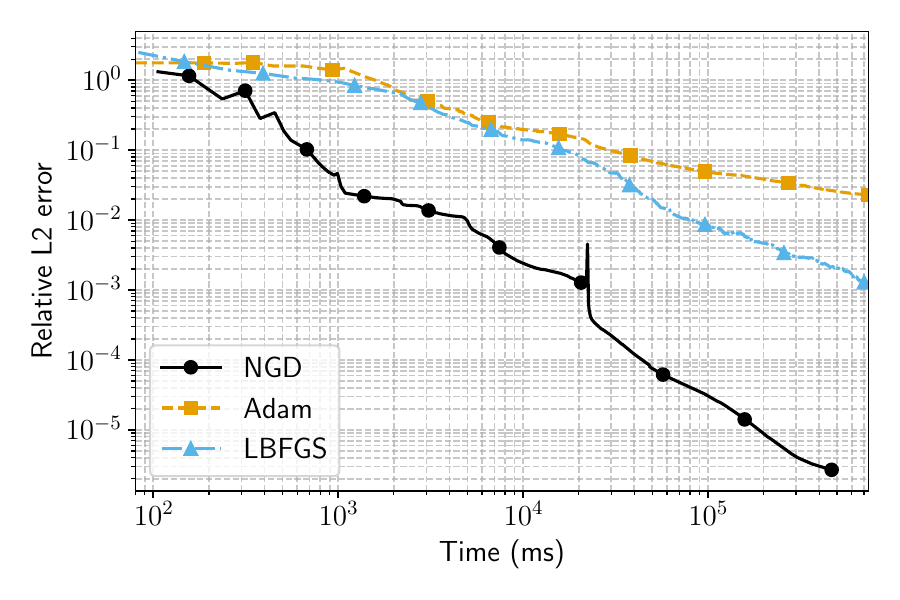}
        \caption{$d=5$}
        \label{fig:recovery-2-time-d5}
    \end{subfigure}
    \caption{Relative $L^2$ error versus time for the optimizers Adam, NGD and L-BFGS for the recovery problem \eqref{eq:recovery-2} in log-log scale, for dimensions $d=4,5$.}
    \label{fig:recovery-2-time}
\end{figure}

The plots in \cref{fig:recovery-2-time} show the evolution of the relative $L^2$ error versus the time spent for the optimisation algorithms. They show that the NGD is a bit faster than Adam for an iteration, and that, despite taking more time than L-BFGS, it converges faster.

\subsection{Layers in low-rank format}


In this subsection, we study the properties and the behaviour of the Gram matrix $G(\theta)$ in order to check if the low-rank solvers like AMEn or (M)ALS can be used. A well-known major challenge with NGD is the occurring rank-defficiency. We also explore some regularization techniques to make these solvers exploitable.

\subsubsection{On the conditioning of $G(\theta)$}

We propose to study the Gram matrix $G_\ell(\theta)$ of each layer $1 \leq \ell \leq L$. To do so, we track the \emph{effective condition number}
\begin{equation*}
    \kappa_\ell(\theta) := \|G_\ell(\theta)\|_{2 \to 2} \|G_\ell(\theta)^\dagger\|_{2 \to 2}
\end{equation*}
during the optimisation of the problem \eqref{eq:recovery-2} in dimension $d=4$. We also track the rank of the Gram matrices $G_\ell$. The experiment is repeated $25$ times.

\begin{figure}[h]
    \centering
    \begin{subfigure}{0.49\linewidth}
        \centering
        \includegraphics[width=\linewidth]{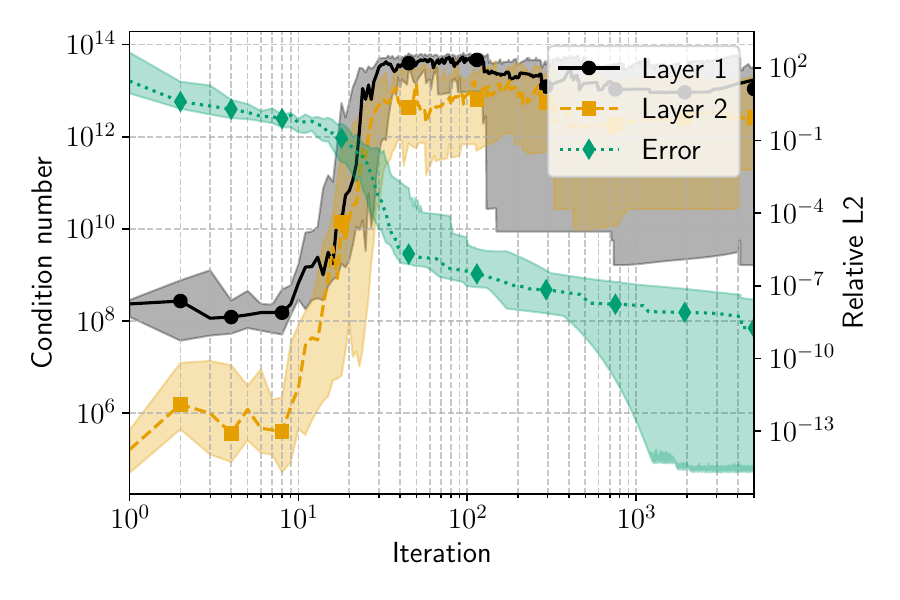}
        \caption{Condition number $\kappa_\ell$.}
        \label{fig:recovery-2-conditioning-d4}
    \end{subfigure}
    \begin{subfigure}{0.49\linewidth}
        \centering
        \includegraphics[width=\linewidth]{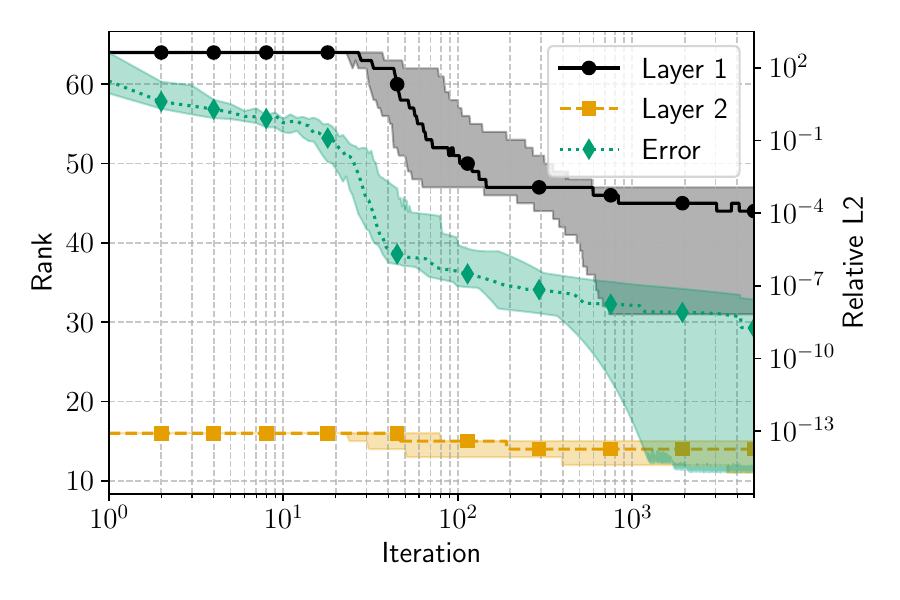}
        \caption{Rank of $G_\ell$.}
        \label{fig:recovery-2-rank-d4}
    \end{subfigure}
    \caption{Condition number and rank of $G_\ell$, with $\ell=1,2$, for the recovery problem \eqref{eq:recovery-2} for dimension $d=4$. The plain line is the median, and the envelope corresponds to the interquartile.}
    \label{fig:recovery-2-conditioning-rank}
\end{figure}

\cref{fig:recovery-2-conditioning-d4} shows that the Gram matrices become highly ill-conditioned during optimization, with condition numbers $\kappa_\ell$ ranging from $10^6$ up to $10^{14}$. In fact, the condition number increases rapidly, reaching values around $10^{13}$ after approximately $20$ iterations. Initially, directions associated with small eigenvalues play a useful role by guiding the optimizer toward a good configuration. However, as the solution approaches optimality, these directions contribute progressively less to the reduction of the loss.

Furthermore, as illustrated in \cref{fig:recovery-2-conditioning-rank}, the rank of the Gram matrices decreases over the iterations. This suggests  that the effective feature space collapses onto a lower-dimensional subspace: the span of active feature directions shrinks, and the remaining features become increasingly correlated. Consequently, the Gram matrix becomes closer to being low-rank.

\subsubsection{Random sketching}

\begin{table}[h]
    \centering
    \begin{tabular}{c|cc}
        Sketching size ($s$) & Learning rate ($\alpha$) & Regularization parameter ($\lambda$)\\
        \hline
        $20$ & $0.85$ & $10^{-11}$\\
        $30$ & $0.8$ & $10^{-12}$\\
        $40$ & $0.7$ & $10^{-12}$\\
    \end{tabular}
    \caption{Hyperparameters for each sketching size $s$.}
    \label{table:hyperparameters-sketching}
\end{table}

As discussed in \cref{sec:ngd}, one way to circumvent the ill-conditioned $G_\ell$ is to project it onto some random subspace.
We apply this randomized method to the recovery problem defined in \eqref{eq:recovery-2}, considering different choice of sketching size $s$. In our experiments, the dimension is fixed to $d=4$, resulting in $m=4 \cdot 2^4=64$ parameters per layer. {For each choice of $s$, the learning rate $\alpha$ and regularization parameter $\lambda$ are adjusted accordingly and heuristically as summarized in \cref{table:hyperparameters-sketching}}

\begin{figure}
    \centering
    \includegraphics[width=0.7\linewidth]{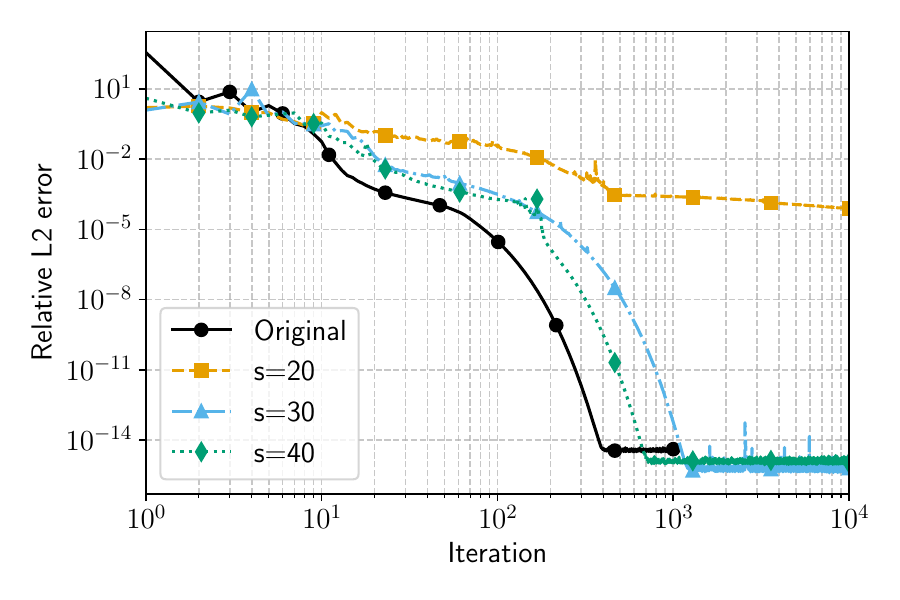}
    \caption{Random sketching for the problem \eqref{eq:recovery-2} in dimension $d=4$.}
    \label{fig:recovery-2-sketching}
\end{figure}

\cref{fig:recovery-2-sketching} shows that around $30$ components of eigenspace of $G_\ell$ is enough to convergence to an optimal solution, which is less than half of the number of eigendirections. However, taking only $20$ or $25$ components is not enough. This plot also confirms that the more components we keep, the better the convergence is. One can remark that taking less than $64$ eigendirections slows the convergence to an optimum.

\cref{fig:recovery-2-sketching} illustrates the convergence behavior for various sketching sizes. We can observe that retaining a rank-$30$ approximation of the Gram matrix  is sufficient to achieve convergence to an optimal solution, which is less than half of the total eigendirections. In contrast, retaining a rank $20$ or $25$   is inadequate for convergence. Moreover, the results confirm that increasing the rank improves convergence, although taking fewer than the full $64$ eigendirections slows the rate at which the solution approaches the optimum.

We also want to point out that \cref{fig:recovery-2-rank-d4} shows that the ranks of $G_\ell$ are not the same for each layer, and that in general the shallowest layer has a smaller rank. This suggests that the sketching size $s$ should be adapted for each layer.

\section*{Acknowledgment}

We acknowledge the partial funding by the ANR-DFG project COFNET (ANR-21-CE46-0015).
This work was partially conducted within the France 2030 framework programme, Centre Henri Lebesgue ANR-11-LABX-0020-01.

\bibliographystyle{abbrv}
\bibliography{ref}

\end{document}